\documentclass{mylms}

\usepackage{amscd}
\usepackage{amssymb}
\usepackage{cite}
\usepackage{enumerate}
\usepackage{amsmath}
\usepackage[usenames, dvipsnames]{color}
\usepackage{hyperref}
\hypersetup{
    colorlinks=true,
    linkcolor=green,
    filecolor=green,
    urlcolor=green,
    citecolor=green
}
\def\w*lim{\mathop{\mbox{\textup{w*-lim}}}}

\usepackage{mathrsfs}

\newtheorem{theorem}{\sc \textbf{Theorem}}[section]
\newtheorem{cor}[theorem]{\sc \textbf{Corollary}}
\newtheorem{prop}[theorem]{\sc \textbf{Proposition}}
\newtheorem{rmk}[theorem]{\sc \textbf{Remark}}

\newtheorem{thm}[theorem]{\sc \textbf{Theorem}}
\newtheorem{lem}[theorem]{\sc \textbf{Lemma}}

\newcounter{cnt1}
\newcounter{cnt2}
\newcounter{cnt3}
\newcounter{cnt4}
\newcommand{\blr}{\begin{list}{$($\roman{cnt1}$)$} {\usecounter{cnt1}
 \setlength{\topsep}{0pt} \setlength{\itemsep}{0pt}}}
\newcommand{\blR}{\begin{list}{\Roman{cnt4}.\ } {\usecounter{cnt4}
 \setlength{\topsep}{0pt} \setlength{\itemsep}{0pt}}}
\newcommand{\bla}{\begin{list}{$($\alph{cnt2}$)$} {\usecounter{cnt2}
 \setlength{\topsep}{0pt} \setlength{\itemsep}{0pt}}}
\newcommand{\bln}{\begin{list}{$($\arabic{cnt3}$)$} {\usecounter{cnt3}
 \setlength{\topsep}{0pt} \setlength{\itemsep}{0pt}}}
\newcommand{\el}{\end{list}}

\newcommand{\cA}{{\mathcal A}}

\newcommand{\cM}{{\mathcal M}}
\newcommand{\cN}{{\mathcal N}}

\newcommand{\cP}{{\mathcal P}}

\newcommand{\cU}{{\mathcal U}}

\newcommand{\extr}{{\rm extr}}

\begin{document}

\title[Resolution of a problem by Luxemburg]{Extreme points of the set of elements majorised by an integrable function: Resolution of a problem by Luxemburg and of its noncommutative counterpart
}

 \author{D. Dauitbek, J. Huang and F. Sukochev}

%\author{D. Dauitbek}
%\address{
%Al-Farabi Kazakh National University, 050040 Almaty, Kazakhstan;
%Institute of Mathematics and Mathematical Modeling, 050010 Almaty, %Kazakhstan.}
%\email{dostilek.dauitbek@gmail.com}
%\author{J. Huang}
%\address{School of Mathematics and Statistics, University of New South Wales, Kensington, NSW, 2052, Australia}
%\email{jinghao.huang@unsw.edu.au}
%\author{F. Sukochev}
%\address{School of Mathematics and Statistics, University of New South %Wales, Kensington, NSW, 2052, Australia}
%\email{f.sukochev@unsw.edu.au}

\classno{46L51; 46L10; 46E30. \hfill Version~: \today}

%\subjclass[2010]{ 46L50; 46M35; 46E30.    }
\extraline{\textbf{Keywords}: spectral scales; probability spaces; extreme points;  majorisation;  noncommutative $L_1$-space; finite von Neumann algebras.
}

%\keywords{spectral scales; finite measure space; extreme points;  majorisation;  noncommutative-$L_1$; finite von Neumann algebras.}\date{}
\maketitle
	{\centering\footnotesize Dedicated to the memory of Professor Wilhelmus Anthonius Josephus Luxemburg (1929--2018)\par}

%\thanks{F. Sukochev and D. Zanin research is supported by the ARC.}

\begin{abstract}
Let $f$ be an arbitrary integrable function on a finite measure space $(X,\Sigma, \nu)$.
We characterise the extreme points of  the set $\Omega (f)$
of all measurable   functions on $(X,\Sigma, \nu)$ majorised by $f$,  providing a complete answer to a problem raised by W.A.J. Luxemburg in 1967.
%showing that the conjecture made by Luxemburg does not hold true in the general setting.
Moreover, we obtain a noncommutative version of this result.
%Precisely, for     an  atomless von Neumann algebra $\cM$ equipped with a faithful normal finite trace $\tau$ and $f\in L_1(\cM,\tau)$,
%we  show that a element  in $\Omega (f) :=\{g : \mbox{$g$ is majorized by  $f$}\}$ is an extreme point if and only if it is equi-measurable with $f$.
%we study extreme points of the orbit of a self-adjoint operator lying in the noncommutative $L_1$-space, associated with finite von Neumann algebra equipped with a faithful normal finite trace $\tau.$ We prove noncommutative analogue of Ryff's theorem that states that extreme points can be described in terms of equimeasurability.
%In addition, we obtain similar results for extreme points of other types of orbits of positive $\tau$-measurable operators.
\end{abstract}

\section{Introduction}
In 1967, W.A.J. Luxemburg  raised  the following question (see \cite[Problem 1]{Luxemburg}):
\begin{quote}
  Determine all the extreme points of $\Omega(f)$,  $f\in L_1(X,\Sigma, \nu)$,  for an arbitrary finite measure space $(X,\Sigma, \nu)$,
  where $\Omega (f)$ is the set of all integrable  functions on $(X,\Sigma, \nu)$ majorised by $f$ in the sense of Hardy--Littlewood--P\'{o}lya.
\end{quote}
The case when $X=\{1,2,\cdots,n \}$ with counting measure  is well-known.
Let $\prec$ be the partial order of Hardy, Littlewood and P\'{o}lya for real $n$-vectors.
It is known that $ y\prec x$ if and only if $y$ belongs to the convex hull $\Omega(x)$ of the set of permutations of $x$ \cite{HLP}, i.e., the convex hull of $\{ Px : P \mbox{ is  a permutation matrix}\}$.
Moreover, the extreme points of $\Omega(x)$ are precisely the permutations of $x$ (see e.g. \cite{HLP}).
If  $(X,\Sigma,\nu)$ is an arbitrary finite measure space and $f,g \in L_1(X)$,
%
%This partial order generalizes to a continuous version, where vectors are replaced  by integrable functions on a finite measure space $(X,\Sigma, \nu)$.
%That is, for functions $f,g \in L_1(X)$,
then $g$ is called \emph{majorised} by $f$ in the sense of Hardy--Littlewood--P\'{o}lya
(denoted by $g\prec f$) provided that  $\int_0^s\lambda(t;g)dt\leqslant\int_0^s\lambda(t;f)dt$
for all $s\in[0,1)$ and
$\int_0^1\lambda(t;g)dt=\int_0^1\lambda(t;f)dt,$ where $\lambda(f)$ is the  right-continuous equimeasurable nonincreasing rearrangement of $f$ (see \cite{C,Ryff63} for the definition, see also Section~\ref{Preliminaries} below).
In the particular case when $(X,\Sigma, \nu)$ is atomless,
the set $\Omega (f):= \{g\in L_1(X): g \prec f\}$ is said to be  the orbit of $f$ (with respect to doubly stochastic operators) \cite{Ryff68,R,Ryff63}.
%For a real-valued Lebesgue integrable function $f\in L_1(0,1)$, we consider the set $  \Omega(f)  := \{ g\prec f \}$, which is the orbit of $f$ with respect to the doubly stochastic operators \cite{Ryff63,R}.
It was first proved by    Ryff \cite{R} that
  if $f,g \in L_1(0,1)$ and $g\sim f$ (i.e. $g\prec f$ and $f\prec g$), then $g$ is an extreme point of $\Omega(f)$.
In 1967, Luxemburg \cite{Luxemburg} extended the result by Ryff \cite{R} to the setting of an arbitrary  atomless  finite measure space  $(X,\Sigma, \nu)$.
However,  the converse implication  was left unresolved  in \cite{R} and \cite{Luxemburg},
 %and raised  the following question (see \cite[Problem 1]{Luxemburg}):
%\begin{quote}
%  Determine all the extreme points of $\Omega(f)$, ~$f=L_1(X,\Sigma, \nu)$ for an arbitrary measure space.
%\end{quote}
and  was
%Luxemburg's problem  when $(X,\Sigma,\nu)$ is atomless  was
later treated  by Ryff \cite{Ryff} (see also \cite{Ryff68}), who proved  the following characterisation.
\begin{quote}
Let  $(X,\Sigma, \nu)$ be a atomless  finite measure space. Let $f=L_1(X,\Sigma, \nu)$ and $g\in \Omega (f)$.
Then, $g$ is an extreme point of $\Omega (f)$ if and only
$g$ is equimeasurable with $f$ (i.e. $g\sim f$).
\end{quote}
%It was shown by Ryff \cite{R} that if $f\in L_1(0, 1)$, then the orbit
%$\Omega(f)$ is weakly compact and so, by the Krein-Milman theorem, the orbit
%$\Omega(f)$ is the weakly  (and hence norm)-closed convex hull of its extreme points.
%Furthermore, Ryff \cite{Ryff} proved that each extreme points of an orbit
%$\Omega(f)$ must be equimeasurable with $f.$
%The goal of this paper is to present a noncommutative analogue of Ryff's theorem
%for extreme points of $\Omega(x)$, $x$ is an element in the noncommutative $L_1$-space $L_1(\cM,\tau)$.
%In this paper,
%we show that Ryff's result indeed gives an affirmative answer to the question by Luxemburg.
%Moreover, we consider question by Luxemburg in a much more general setting.
However, the case of arbitrary finite measure spaces (e.g. $L_1(0,\frac 12)\oplus \mathbb{C}e$) seems to have been left open and this case cannot be obtained as a corollary from the atomless case  (see Remark \ref{R:Ex}).
The main object of the present paper is to provide a complete answer to Luxemburg's  question.
Moreover, we consider this question in a  much more general setting, giving a characterisation
for the extreme  points of the set $\Omega(y)$ of all self-adjoint  operators majorised by a self-adjoint operator $y$ in the  noncommutative $L_1$-space affiliated to
 a finite von Neumann algebra
 (see e.g. \cite{Hiai,HN1,SZ2} for related partial results in the noncommutative setting).  %(see Section \ref{Preliminaries} for definitions).

Let $\mathcal{M}$ be a  von Neumann algebra
equipped with a faithful normal finite trace $\tau$ and $L_1(\mathcal{M},\tau)_h$
(resp. $L_1(\mathcal{M},\tau)_+$) be the set of all self-adjoint (resp. positive)
 operators in the noncommutative $L_1$-space $L_1(\cM,\tau)$.
 Petz \cite{P} introduced the spectral scale $\lambda (x)$ of a $\tau$-measurable self-adjoint operator $x$.
 In the special case when $\cM$ is commutative (and hence, the pair $(\cM,\tau)$ can be identified with $L_\infty (X,\Sigma,\nu)$),   the
 the notion of spectral scales coincides with the non-increasing rearrangements.
 %, which coincides with the non-increasing rearrangement in $L_1(0,1)$.
For detailed discussions of this notion, %and the properties of increasing rearrangements in this setting
 we refer to \cite{HN1} (see also \cite{Hiai,DPS,DDP}).
Theorem \ref{RyffNonCom} below  is the main result of the present paper, which unifies Ryff's theorem  \cite{Ryff,Ryff68} and the classic result for vectors  \cite{HLP} with significant extension.
The following theorem yields the complete resolution of Luxemburg's problem in the general setting.

%In particular, we
 %   completely resolve  Luxemburg's problem.
%
%
%
%  which gives a noncommutative version  (see \cite{Hiai,HN1,SZ,SZ2} for %related partial results in the noncommutative setting) of
%  Moreover,  we consider arbitrary finite measure spaces rather than atomless finite measure spaces,

\begin{thm}\label{RyffNonCom}
Assume that $\cM$ is a von Neumann algebra equipped with a faithful normal tracial state $\tau$.
Let $y\in L_1(\cM,\tau)_h$ and let $\Omega(y) $  be   defined as the set of all self-adjoint  operators $x\in L_1(\cM,\tau)$ satisfying  $\lambda (x) \prec \lambda (y) .$
Then, $x$ is an extreme point if and only if for each  $t\in (0,1)$, one of the following options holds:
\begin{enumerate}
  \item[(1).]   $\lambda(t;x)=\lambda(t;y)$;
  \item [(2).]  $\lambda(t;x) \ne \lambda(t;y)$ with  the spectral projection $ E^x \{\lambda (t;x)  \} $ being an atom in $\cM$ and $$\int_{ \{s;\lambda (s;x)=\lambda (t;x)\}}  \lambda (s;y)ds    =  \lambda(t ;x)     \tau(E^x (\{\lambda (t;x)\}))  .  $$
\end{enumerate}
\end{thm}

Since any  finite measure space is a   von Neumann algebra equipped with a faithful normal finite trace, the complete resolution of Luxemburg's problem is an immediate consequence of Theorem \ref{RyffNonCom}.
Without loss of generality, we state the result for normalized measure spaces.
\begin{cor}\label{cor:Lux}
Assume that $(X,\Sigma, \nu)$ is an arbitrary  normalized measure space.
Let $y\in L_1(X)$ and let $\Omega(y) $ be the set of all integrable  functions on $(X,\Sigma, \nu)$ majorised by  $y$  in the sense of Hardy--Littlewood--P\'{o}lya.
Then, $x$ is an extreme point if and only if for each  $t\in (0,1)$, one of the following options holds:
\begin{enumerate}
  \item[(1).]   $\lambda(t;x)=\lambda(t;y)$;
  \item [(2).]  $\lambda(t;x) \ne \lambda(t;y)$ with  the preimage $  x^{-1}(\lambda (t;x))   $ being an atom in $(X,\Sigma, \nu)$ and $$\int_{ \{s;\lambda (s;x)=\lambda (t;x)\}}  \lambda (s;y)ds    =  \lambda(t ;x)   \nu(x^{-1}(\lambda (t;x)))   .  $$
\end{enumerate}
\end{cor}

\begin{rmk}\label{R:Ex}
Let $f (t):= \frac{1}{\sqrt{t}}$, $t\in(0,\frac 12)$.
 Hence, $f\oplus 0\in L_1(0,\frac 12)\oplus \mathbb{C} e$, where $e$ is an atom with measure $\frac 12$.
Even though $L_1(0,\frac 12)\oplus \mathbb{C}e$ can be embedded into an atomless finite measure space, Ryff's theorem only provides  a sufficient condition for $g\in  L_1(0,\frac 12)\oplus \mathbb{C} e$ to be an extreme point of $\Omega(f)$.
 That is, every  $g$ with $\lambda (g)=\lambda (f\oplus 0e )$ is an extreme point  of $\Omega (f \oplus 0e  )$.
 However, there exist extreme points   of $\Omega(f\oplus 0e )$ which could not be described in this way.
By Corollary~\ref{cor:Lux},
for every $a\in [0,\frac{1}{2})$, $g := f\chi_{(0,a)}  \oplus 2 e \int_a^{\frac{1}{2}} f(t) d t$ is an extreme point of $\Omega (f\oplus 0e)$.
In particular, $0\oplus 2 \sqrt {2}e$ is an extreme point of $\Omega (f\oplus 0e)$.
 \end{rmk}

Let $M_n (\mathbb{C})$ be the $n\times n$ matrices and $E_D$ be the compression onto $D$, the diagonal masa (see e.g. \cite{SS}) in $M_n (\mathbb{C})$.
The celebrated Schur-Horn theorem \cite{Horn,Schur} sates that
$$E_D (\{U M_\beta U^*: U\in M_n (\mathbb{C})\mbox{ is unitary}\}) = \{M_\alpha \in D : \alpha \prec \beta \},~ \beta\in M_n (\mathbb{C})_h, $$
where $M_\alpha$ is the diagonal matrix with the entries of $\alpha$ in the main diagonal.
Inspired by the Arveson--Kadison conjecture \cite{AK},
 several authors (e.g. \cite{AM,Jasper,BJ,BJS,MR,KS}) have worked on the analogues to the Schur--Horn theorem.
In particular,  Argerami and Massey \cite{AM}  established a result for type  II$_1$ factors.
As an application of Theorem \ref{RyffNonCom}, by applying results in \cite{Horn,Schur} and \cite{AM},  we obtain the following.

\begin{cor}\label{corsh}
Given a type II$_1$ factor $\cM$  (resp. $\cM=M_n(\mathbb{C})$) and a diffuse abelian von Neumann subalgebra $\cA\subset  \cM$ (resp. $\cA$ is the  diagonal masa in $M_n (\mathbb{C})$), for every $y\in \cM_h$,  we have
$$\extr (\overline{E_\cA (\cU_\cM(y))}^{\sigma \text{-sot}} )= \extr(\{ x\in \cA_h : x\prec y \}) = \{  x\in \cA_h : \lambda (x)=\lambda (y)\}   ,$$
where $E_\cA$ is the conditional expectation onto $\cA$ and $\cU_\cM(y)$ is the unitary orbit of $y$ in $\cM$.
\end{cor}
In particular, by the  Krein--Milman theorem,    $\overline{E _\cA (\cU_\cM(y))}^{\sigma \text{-sot}} $ is the  closed convex hull of $\{  x\in \cA_h : \lambda (x)=\lambda (y)\}$ in the $\sigma$-strong operator topology.

%We now briefly describe the structure of the paper. In Section~\ref{Preliminaries}
%we present all necessary notations and definitions.
%In Section~\ref{secCom} we
%present a new proof of
%Ryff's theorem in the commutative setting \cite{Ryff}.
% Our approach appears to be more streamlined than the original method of Ryff (see \cite{Ryff}) and lends itself to a ready noncommutative extension. In the final section we prove the main results of the paper on  Luxemburg's question  in the noncommutative setting.

Finally, we comment briefly at  the interconnection of spectral scales and singular value functions.
For a given positive $\tau$-measurable operator $x$ affiliated to $\cM$, the spectral scale $\lambda(x)$ of the operator is equal to the singular values function $\mu(x),$ see e.g. \cite{FK}. However, for a general self-adjoint $x\in L_1(\mathcal{M},\tau)_h$ the spectral scale $\lambda(x)$ and singular value function $\mu(x)$ are different, and so one can consider other majorisations given by singular value functions. Namely, we say that $y$ is submajorised by $x$ (denoted by  $y\prec\prec x$) if $\int_0^s\mu(t;y)dt\leqslant\int_0^s\mu(t;x)dt$ for all $s\in[0,\tau(\mathbf 1)).$
In \cite{S85} and \cite{Sukochev} (see also \cite{CKS}), it is  proved that for an arbitrary atomless finite von Neumann algebra with a faithful normal finite trace $\tau$ and    for any $x\in L_1(\mathcal{M},\tau)_h$,
the extreme points of    $\{y\in L_1(\mathcal{M},\tau)_h\ :\ y\prec\prec x\}$ are exactly the operators $y$ satisfying $\mu(y)=\mu(x)$ (see  \cite{CK1, CK2, CK3, SZ2,Hiai} for more related results).
  The result stated in Theorem~\ref{RyffNonCom} does not follow from that of \cite{CKS} even in the setting of atomless von Neumann algebras.
  Our proof of Theorem~\ref{RyffNonCom} is completely different  from that in  \cite{CKS} and is based on a careful study of Hardy--Littlewood--P\'{o}lya majorisation in the commutative and the noncommutative setting.% to noncommutative realm.
\section{Preliminaries}\label{Preliminaries}

Throughout this paper, we denote by $\mathcal{M}$ a von Neumann algebra equipped with a faithful normal finite trace $\tau$.
We  denote by $\mathbf 1$ the
identity in $\mathcal{M}$ and by $\mathcal{P}(\mathcal{M})$ the collection  of all orthogonal projections in $\mathcal{M}$.
Without loss of generality, we assume that $\tau({\bf 1})=1$.
A densely defined closed linear operator $x$ affiliated
with $\mathcal{M} $ is called \emph{$\tau$-measurable} if for each $\varepsilon>0$ there exists $e\in\mathcal{P}(\mathcal{M})$ with $\tau(e^{\perp})\leq\varepsilon$ such that $e(\mathcal{H})\subset\mathfrak{D}(x)$. Let us denote by $S(\mathcal{M},\tau)$ the set of all $\tau$-measurable operators.
We note that, since we assume that  the trace $\tau$ is finite, the space $S(\cM,\tau)$ is the set of all densely defined closed linear operators $x$ affiliated
with $\mathcal{M} $.
However, if the trace $\tau$ is infinite, then there are densely defined closed linear operators which are not $\tau$-measurable.
The set of all self-adjoint elements in $S(\mathcal{M},\tau)$ is denoted by $S(\mathcal{M},\tau)_h$,
which is a real linear subspace of $S(\mathcal{M},\tau)$.
The set of all positive elements in $S(\mathcal{M},\tau)_h$ is denoted by $S(\mathcal{M},\tau)_+$.

If $x\in S(\mathcal{M},\tau)_h$, then the \emph{spectral distribution function} $d(x)$ of $x$ is defined by setting
$$d(s;x)=\tau(E^ x(s,\infty)),\ \ s\in\mathbb{R},$$
where $E^x(s,\infty)$ is spectral projection of $x$ on the interval $(s,\infty).$

It is clear that the function $d(x):\mathbb{R}\rightarrow[0,\tau(\mathbf 1)]$ is decreasing and the normality of the trace implies that $d(x)$ is right-continuous.
If $x\in S(\mathcal{M},\tau)$, then the \emph{singular value function} $\mu(x):[0,1)\rightarrow[0,\infty]$ \cite{FK, F, DPS} is defined to be the decreasing right-continuous inverse of the spectral distribution function $d(|x|)$, that is,
$$\mu(t;x)=\inf\{s\geqslant0\ :\ d(s;|x|)\leqslant t\},\ \ t\in[0,1).$$

We introduce the notion of spectral scales (see  \cite{P}, see also \cite{HN1, HN2, DDP, DPS,AM}).
If $x\in S(\mathcal{M},\tau)_h$, then the \emph{spectral scales} (also called \emph{eigenvalue functions}) $\lambda(x):[0,1)\rightarrow(-\infty,\infty]$ and $\check{\lambda}(x):[0,1 )\rightarrow(-\infty,\infty]$ are defined by
$$\lambda(t;x)=\inf\{s\in\mathbb{R}\ :\ d(s;x)\leqslant t\},\ \ t\in[0,1),$$
and
$$\check{\lambda}(t;x)=\sup\{s\in\mathbb{R}\ :\ \tau(E^x(-\infty,s])\leqslant t\},\ \ t\in[0,1).$$
The spectral scales $\lambda(x)$ (resp. $\check{\lambda}(x)$) are decreasing (resp. increasing) right-continuous functions.
It is an immediate  corollary of \cite[Proposition 1]{P} (see also \cite[Chapter III, Remark 5.4]{DPS}) that \begin{align}\label{disx}
d(x)=d(\lambda (x)).
\end{align}
If $x\in S(\mathcal{M},\tau)_+$, then it is evident that $\lambda(t;x)=\mu(t;x)$ for all $t\in[0,1)$.
Note \cite{DPS,DDP} that
$$\check{\lambda}(t;x)=\lambda((1 -t)-;x)=-\lambda(t;-x),\ \ \forall t\in[0,1 ), $$
where $\lambda((\tau(\mathbf 1)-t)-;x) = \lim_{\varepsilon  \rightarrow 0^+} \lambda((\tau(\mathbf 1)-t)-\varepsilon;x)$.

Assume that  $\mathcal{M}=L_{\infty}(0,1)$ and $\tau(f)=\int_{0}^1 fdm$,
 where  $m$ the   Lebesgue measure on $(0,1)$.
 In this case, $S (\mathcal{M},\tau)_h$ consists of all real measurable functions $f$ on $(0,1)$. For every $f$, $\lambda(f)$ coincides with  the \emph{right-continuous equimeasurable nonincreasing rearrangement} $\delta_f$ of $f$ (see e.g. \cite{HN1}):
$$\lambda(t;f)=\delta_f(t)=\inf\{s\in\mathbb{R}:\ m(\{x\in X:\,f(x)>s\})\leqslant t\},\ \ t\in[0,1 ).$$

Using the extended trace $\tau:S(\mathcal{M},\tau)_+ \rightarrow[0;\infty]$ to a linear functional on $S(\mathcal{M},\tau),$ denoted again by $\tau$, the noncommutative $L_1$-space (written by $L_1(\mathcal{M},\tau)$) is defined by
$$L_1(\mathcal{M},\tau)=\{x\in S(\mathcal{M},\tau):\tau(|x|)<\infty\}.$$
(see e.g. \cite[p.84]{D}). Let us denote
$L_1(\mathcal{M},\tau)_h=\{x\in L_1(\mathcal{M},\tau):x=x^{\ast}\}$ and $L_1(\mathcal{M},\tau)_+=L_1(\mathcal{M},\tau)\cap S(\mathcal{M},\tau)_+ .$
We note that for every $x\in L_1(\cM,\tau)_h$, we have (see e.g. \cite{DDP}, \cite[Proposition 1]{P} and \cite[Chapter III, Proposition 5.5]{DPS})
\begin{equation}\label{spectral scale}
  \tau(x)=\int_0^1\lambda(t;x)dt=\int_0^1 \check{\lambda}(t;x)dt.
\end{equation}

%In commutative case, $L_1(0,m(X))$ is the set of all real-valued Lebesgue measurable functions on interval $(0,m(X))$ and $L_1(0,m(X))_+$ is the set of all positive Lebesgue measurable functions in $L_1(0,m(X)).$

%A semigroup of absolute contractions (or admissible operators) and substochastic operators
%are defined by
%$$
%\Sigma':=\{T:L_1(0,1)\to L_1(0,1):\max\left(\|T\|_{L_1\to L_1},\|T\|_{L_{\infty}\to L_{\infty}}\right)\leqslant1\}
%$$
%(see
%e.g. \cite[II.3.4]{KPS})
%and
%$$
%\Sigma'_{+}:=\{0\leqslant T\in\Sigma'\}
%$$
%(see e.g. \cite[p. 107]{BS}), respectively.
%We define a semigroup of doubly stochastic operators as follows
%$$
%\Sigma:=\{T\in\Sigma'_+:\,m(Tf)=m(f),\,\,\forall f\in L_1(0,1), \,\,T(\mathbf 1)= \mathbf 1\}
%$$
%and
%$$
%\Sigma_{+}:=\{T\in\Sigma'_+:\,m(Tf)=m(f),\,\,\forall f\in L_1(0,1)_+, \,\,T(\mathbf 1)= \mathbf 1\}
%$$
%(see e.g. \cite{R}), where $m(f)=\int f(s)dm(s)$, $m$ is Lebesgue measure.

%%If $f\in L_1(0,1)$ (or $f\in L_1(0,1)_+$), then for $\Sigma$ (resp. $\Sigma',$ $\Sigma'_{+}$ and $\Sigma_{+}$) we denote by
%$\Omega(f)$ (resp. $\Omega'(f)$, $\Omega'_{+}(f)$ and $\Omega_{+}(f)$) the \emph{orbit} of $f$ with respect to the semigroups $\Sigma$ (resp. %$\Sigma',$ $\Sigma'_{+}$ and $\Sigma_{+}$).

For every $f,g\in L_1(\cM,\tau),$ $g$ is said to be  \emph{majorised} by $f$ in the sense of Hardy--Littlewood--P\'{o}lya (written by $g\prec f$)  if
$$\int_0^s\lambda(t;g)dt\leqslant\int_0^s\lambda(t;f)dt$$
for all $s\in[0,1)$ and
$$\int_0^{1}\lambda(t;g)dt=\int_0^{1}\lambda(t;f)dt.$$
$g$ is said  to be  \emph{submajorised} by $f$ in the sense of Hardy--Littlewood--P\'{o}lya (written by $g\prec\prec f$) if and only if
$$\int_0^s\mu(t;g)dt\leqslant\int_0^s\mu(t;f)dt, ~\forall s\in[0,1).  $$
For  a self-adjoint element $y\in L_1(\mathcal{M},\tau)_h$, we denote
$$\Omega(y):=\{x\in L_1(\mathcal{M},\tau)_h\ :\ x\prec y\}.$$
We note  that
\begin{align}\label{tri}
\lambda (x+y)\prec \lambda (x)+\lambda (y), ~\forall x,y\in L_1(\cM,\tau).
\end{align}
For every $x\in S(\cM,\tau)_h$ and  $s\in \mathbb{R}$, if $e\in \cP(\cM)$ is such that $$E^x(s,\infty) \le e \le E^x [ s,\infty ),$$
then  (see \cite{DDP} or \cite[Chapter III, Lemma 8.2]{DPS}):
\begin{align}\label{2.2}
\lambda (t;xe) =\lambda (t;x) \mbox{ for all } t\in [0,\tau(e)),
\end{align}
and
\begin{align}\label{2.3}
\lambda (t;xe^\perp ) =\lambda (t+\tau(e);x) \mbox{ for all } t\in [0,\tau(e^\perp )).
\end{align}
Suppose that $x,y\in S(\cM,\tau)_h$.
If $\lambda(|x|)\lambda(|y|)\in L_1(0,1)$, then (see \cite[Proposition 2.3]{DDP},  see also \cite[Chapter III, Theorem 6.6]{DPS}) $xy \in L_1(\cM,\tau)$ and
\begin{align}\label{2.4}
\int_0^1 \lambda (s;x)\check{ \lambda} (s;y)ds \le \tau(xy)\le \int_0^1 \lambda (s;x)\lambda (s;y)ds.
\end{align}
If, in addition,  $\cM$ is atomless, then for any $x\in L_1(\cM,\tau)$ and $t\in [0,1)$, we have (see e.g. \cite[Proposition 1.1]{HN1} and \cite[Chapter III, Remark 8.6]{DPS})
\begin{align*}
\int_0^t \lambda (s;x)ds & =\sup \{\tau(xe) :e \in \cP(\cM), \tau(e)=t\} \\
&=\max\{ \tau(xe);e\in \cP(\cM), \tau(e) =t , ex=xe\}.
\end{align*}

%%%For a function $f\in L_1(0, 1)$,
%%%we define:
%%%\begin{itemize}
%%%  \item $\Omega(f):=\{g\in L_1(0,1)\ :\ g\prec f\}$,
%%% {\color{red} \item $\Omega_{+}(f):=\{g\in L_1(0,1)_+\ :\ g\prec\prec f\ \text{and}\ m(g)=m(f)\}$,
%%%  \item $\Omega'_{+}(f):=\{g\in L_1(0,1)_+\ :\ g\prec\prec f \}.$}
%%%\end{itemize}
%%% We note that $g\prec f$    if and only if there exists a doubly stochastic  operator
%%%$T$ such that  $g=Tf$  \cite[Theorem 3]{R}.

%%%In noncommutative case, given $x, y\in L_1(\mathcal{M},\tau)_h$, we say $y$ is \emph{majorised} (resp. \emph{submajorised}) by $x$ in the sense of Hardy, Littlewood and P\'{o}lya written $y\prec x$ (resp. $y\prec\prec f$) if and only if $\lambda(y)\prec\lambda(x)$
%%%(resp. $\mu(y)\prec\prec\mu(x)$).

%%%%{\color{red}For a positive element $x\in L_1(\mathcal{M},\tau)_+$, we denote
%%%%$$\Omega_{+}(x):=\{y\in L_1(\mathcal{M},\tau)_+\ :\ y\prec\prec x\ %%%%\text{and}\ \tau(y)=\tau(x)\}$$
%%%%and
%%%%$$\Omega'_{+}(x):=\{y\in L_1(\mathcal{M},\tau)_+\ :\ y\prec\prec %%%%x\}.$$}

\section{Some technical results}\label{secCom}

%The following proposition was firstly proved by J.V. Ryff on finite atomless measure space (see \cite[Lemma 2]{R}). See also the exposition in \cite[Theorem 7.5, pp. 82-83]{BS}.

%\begin{prop}\label{measure-preserving}
%Let $(X,m)$ be  finite Lebesgue measure space. If $f\in L_1(0,m(X)),$ then there exists a measure-preserving transformation $\gamma:X\to[0,m(X)) $ such that $f=\lambda(f)\circ\gamma.$

%\end{prop}
Let $\cM$ be a von Neumann algebra equipped with a finite faithful normal trace.
Some of the results in this section are well-known for positive operators (see e.g. \cite{DPS,HSZ,CKSa,CS}). However, the results in this section do not follow from those for positive operators.

%%%%%%%%%%\begin{lem}\label{concatenation lemma}
%%%%%%%%%% Let $y_1$ and $y_2$ be mutually disjointly supported operators in $L_1(\cN,\tau)_h$.
%%%%%%%%%% Assume that $g_1 \in \Omega (y_1)$ (respectively, $g_2 \in \Omega (y_2)$) with $\s(g_1) \le \s(y_1) $ (respectively, $\s(g_2) \le \s(y_2) $).
%%%%%%%%%% Define $y:=y_1+y_2$ and $g:=g_1+g_2$.
%%%%%%%%%% If $g\in \extr (\Omega (y))$, then $g_i \in \extr (\Omega(y_i))$, $i=1,2$.
%%%%%%%%%%\end{lem}
%%%%%%%%%%\begin{proof}
%%%%%%%%%%   It suffices to consider the case for $g_1$.
%%%%%%%%%%   Assume that  $g_1=\frac12(g_{11}+g_{12})$ with $g_{1j}\prec y_1$.
%%%%%%%%%%    Let $$h_j := g_{1j} +g_2  ,   ~j=1. 2$$

%%%%%%%%%% By the assumptions that  $g_2\in\Omega(y _2),\ g_{1j}\prec y_1$ and Proposition  \ref{concatenation}, we obtain that for  every $s\in [0,1 ]$,
%%%%%%%%%%\begin{eqnarray*}
%%%%%%%%%%  \int_0^s\lambda(t;h_j)dt &=& \sup_{a+b=s}\int_0^a\lambda(t;g_{1j})dt+\int_0^b\lambda(t;g_2)dt \\
 %%%%%%%%%%  &\leqslant& \sup_{a+b=s}\int_0^a\lambda(t;y_1)dt+\int_0^b\lambda(t;y_2)dt =  \int_0^s\lambda(t;y)dt.
%%%%%%%%%%\end{eqnarray*}
%%%%%%%%%%That is, $h_j\prec y$.
%%%%%%%%%% Hence, $g=\frac12(h_1+h_2)$ with $h_i\in\Omega(y).$ By the assumption $g$ is an extreme point of $\Omega(y)$,
%%%%%%%%%%  we infer that $h_1=h_2$.
%%%%%%%%%%  Hence  $g_{11}=g_{12}$ and therefore, $g_1$ is an extreme point.
%%%%%%%%%%\end{proof}
%, however, our proof is rather different from that of \cite{Ryff}.
We note that for any $y\in L_1(0,1)$, the notation $\Omega(y):=\{ x\in L_1(\cM,\tau):\lambda (x)\prec y  \}$  makes sense.
All results in the present section hold true for $y\in L_1(0,1)$ and $x\in L_1(\cM,\tau)$.
However, to avoid ambiguity, we always assume that $y\in L_1(\cM,\tau) $.

The following is a noncommutative   analogue of Ryff's Proposition stated in \cite{Ryff}.
\begin{lem}\label{counter example 1} Let $y  \in L_1(\cM, \tau)_{h}.$ If $x \in\Omega(y)=\{x\in L_1(\mathcal{M},\tau)_h:\ x\prec y \}$ satisfies that
\begin{enumerate}[{\rm (i)}]
  \item $\lambda (s_{i+1};x )< \lambda(s_i;x)$ for some $0<s_i<s_{i+1}<1,$ $i=1,2,3$,
  \item $ \int_0^{s_1}\lambda (t;x  )dt+\lambda (s_1;x )(s-s_1) \le \int_0^{s}\lambda (t; y)dt ,$ for all $s\in[s_1,s_4],$
\end{enumerate}
then $x \not\in\emph{extr}(\Omega(y)).$
\end{lem}

\begin{proof}
For the sake of convenience, we denote $a_i =\lambda(s_i;x )$, $i=1,2,3,4$.
Let $p_1 := E^{x } [\frac{a_2+a_3}{2}, \frac{a_1+a_2}{2})$ and $p_2:= E^{x } [\frac{a_3+a_4}{2},\frac{a_2+a_3}{2})$.
We denote $T_1= \tau(p_1 )$ and $T_2 = \tau(p_2) $.
Observe that  $T_1,T_2\in (0,1)$ and $p_1p_2 =0 $.
Set
$$ u : =  p_1 -  \frac{T_1}{T_2} p_2 . $$
It is clear that $\tau(u) =0$.
Assume that $\delta>0$ such that $$\delta < \min\{ \frac{a_1-a_2}{2}, \frac{(a_3-a_4)T_2 }{2 T_1} , \frac{ (2a_1-a_2-a_3)T_2 }{2T_1} , \frac{a_2+a_3-2a_4}{2}\} .$$
Let $x_\pm=x\pm \delta u$.
By the spectral theorem, $\lambda(s;x_{\pm})=\lambda (s;x )$ for
$s\not\in[s_1,s_4]$ and $\lambda(s;x_{\pm})\leqslant \lambda (s_1;x  )$ for
$s\in[s_1,s_4].$
We assert that $x_{\pm}\prec y $.

%Set $s_5=\frac12(s_2+s_3),$
%$u=\chi_{[s_2,s_5]}-\chi_{[s_5,s_3]}$ and $g_{\pm}=g\pm\delta u$ for some $\delta>0.$

% for sufficiently small $\delta$.
%In particular,
% Since
%$$\int_{s_2}^{s_3}g_{\pm}(s)ds=\int_{s_2}^{s_3}g(s)ds,$$
Note that
\begin{eqnarray*}
\int_0^{1}\lambda(t;x_{\pm})dt =
 \tau (x_{\pm})    =\tau( x \pm u)=\tau(x )=
 \int_0^{1}\lambda(s;x )ds= \int_0^{1}\lambda(t;y )dt.
\end{eqnarray*}
Since $\lambda(s;x_{\pm})=\lambda (s;x)$ for
$s\not\in[s_1,s_4]$,  it may be concluded that
$$\int_{s_1}^{s_4}\lambda(t;x_{\pm})dt-\int_{s_1}^{s_4}\lambda (t;x) dt=\int_0^{1}\lambda(t;x_{\pm})dt-\int_0^{1}\lambda (t;x) dt=0.$$
Hence,
$$\int_0^s\lambda(t;x_{\pm})dt=\int_0^s \lambda (t;x)dt\leqslant\int_0^s \lambda (t;y )dt,\quad s\not\in[s_1,s_4],$$
where the last inequality follows from the assumption that  $x\in\Omega(y).$
On the other hand, since $\lambda (x)$ is decreasing, it follows that
$$\int_0^{s}\lambda(t;x _{\pm})dt\leqslant\int_0^{s_1}\lambda (t;x )dt+ \lambda (s_1;x)(s-s_1)\stackrel{(ii)}{\leqslant }\int_0^{s}\lambda (t; y )dt,\quad s\in[s_1,s_4].$$
Hence, $x_{\pm}\in\Omega(y )$ and $x =\frac12(x_++x _-).$
That is, $x \not\in\emph{extr}(\Omega(y))$.
\end{proof}

\begin{lem}\label{counter example 2} Let $y\in L_1(\cM,\tau).$ If $x  \in\emph{extr}(\Omega(y))$ and $\int_0^s \lambda (t;y )dt>\int_0^s\lambda (t;x)dt$ for all $s\in(t_1,t_2),$ then $\lambda (x) $ is a step function on $(t_1,t_2)$.
\end{lem}
\begin{proof}
Assume by contradiction that there exists $s_1\in (t_1,t_2)$   such that  $(s_1-\varepsilon,s_1)$ and $(s_1,s_1+\varepsilon)$ are not constancy intervals of $\lambda (x )$ for some $\varepsilon>0.$ Define
$$N=\max{\left( -\lambda (s_1;y), 0\right)}  +\max{\left( -\lambda(s_1;x), 0 \right)}  +1.$$
Let $Y= y+N{\bf1} $ and $X=x+N{\bf1}  $.
Since  $x\in \extr(\Omega (y))$, it follows that $X\in \extr(\Omega (Y))$.
Moreover, by spectral theorem (see also  \cite[Remark 6.1 (5)]{Hiai} or \cite[Chapter III, Proposition 5.7]{DPS}), we have  $\lambda (X)=\lambda (x+N{\bf1}) =\lambda (x)+N$ and $\lambda (Y)= \lambda (y+N{\bf1}) =\lambda (y)+N$.
Therefore,
$$\int_0^s \lambda (t;Y)dt>\int_0^s\lambda (t;X)dt  , s\in (t_1,t_2) .$$
%Define
%$$F(t)=f(t)+N,\ \ G(t)=g(t)+N, \ \ \forall t\in(0,1).$$
Note that  $\lambda (s_1;Y),\lambda (s_1;X)>0$.
Since $\lambda (Y)$ and $\lambda (X)$ are decreasing functions, it follows that $\lambda (Y)$ and $\lambda (X) $ are strictly positive on $(0 ,s_1]$.

%Note that  $F=\lambda(F)\in L_1(0,1),\ G=\lambda(G)\in L_1(0,1)$ and $$\int_0^{s_1}(F(t)-G(t))dt=\int_0^{s_1}(f(t)-g(t))dt>0.$$
Since $\lambda (Y)$ and $\lambda (X)$ are right-continuous,
one can choose $s_2>s_1$ such that  $\lambda (s_2;Y)>   0$ and  $\int_0^s \lambda (t;Y)-\lambda (t; X)dt >  0$ for every  $s \in (t_1 ,s_2)$ and
$$\frac{\int_0^{s_1}\lambda (t;Y)-\lambda (t;X)dt}{\lambda (s_1;X)}\ge s_2-s_1.$$
Hence, for any $s\in [s_1,s_2]$, we have
$$ \int_0^{s} \lambda (t;Y)dt \ge   \int_0^{s_1}\lambda (t; Y )dt \ge   \int_0^{s_1} \lambda (t;X)dt +  \lambda (s_1;X) ( s_2-s_1) .  $$
Since $Y$ and $X$ satisfy the assumptions in  Lemma~\ref{counter example 1}, it follows that $X\not\in\text{extr}(\Omega(Y)).$
Consequently, $x\not\in\text{extr}(\Omega(y )).$
\end{proof}
Let $x \in L_1(\cM,\tau)_h$. Denote by $\cN$ the abelian von Neumann algebra generated by all spectral projections of $x $.
In particular, $x\in L_1(\cN,\tau)_h$.
Recall that (see \cite[Proposition 1.1]{HN1})
\begin{align}\label{Hineq}\int_0^s   \lambda (t;x)dt =\sup \{\tau(xa);a\in \cM, 0\le a\le 1, \tau(a)=s\}.
\end{align}
Assume that $\tau(xe) = \int_0^s  \lambda (t;x)dt$ for a projection $e\in \cM$ with $\tau(e)=  s $.
Let $E_\cN$ be  a  conditional expectation from $L_1(\cM,\tau)$ onto $ L_1(\cN,\tau)$ (see e.g. \cite{U}, see also \cite[Proposition 2.1]{DDPS}).
In particular,  $E _\cN(e)\le {\bf 1} $ and $\tau(E_\cN(e))=\tau(e)=s$.
Moreover, \begin{align}\label{ecnex}
\tau(E_\cN(e)x) =\tau(ex) =\int_0^s \lambda (t;x)dt.
\end{align}
We note that $E_\cN(f)$, $f\in \cP(
\cM)$,  is not necessarily a projection \cite{Sukochev}.
%It is easy to get the following proposition.
\begin{prop}Under the above assumptions on $e$, we have
\begin{align}\label{te}
E^{x}(\lambda (s;x),\infty) \le E_\cN(e) \le E^{x}[\lambda (s;x),\infty)
.\end{align}
\end{prop}
\begin{proof}
We present the proof for  the first inequality and a similar argument yields that $E_\cN(e) \le E^x [\lambda (s;x),\infty)$.

Without loss of generality, we may assume that $E^{x}(\lambda (s;x),\infty)\ne 0$.
Since $\cN$ is a commutative algebra and $0\le E_\cN(e)\le 1 $, it follows that
$$E_\cN(e) E^x (\lambda (s;x), \infty)= E^x (\lambda (s;x), \infty)^{1/2} E_\cN(e)  E^x (\lambda (s;x), \infty) ^{1/2} \le  E^x (\lambda (s;x), \infty).  $$
 If $E_\cN(e)E^x(\lambda(s;x),\infty) = E^x(\lambda(s;x),\infty)$, then \begin{align}\label{ENE>LAMBDA}
 E_\cN(e) \ge E_\cN(e)^{1/2} E^x(\lambda(s;x),\infty)E_\cN(e)^{1/2} = E_\cN(e) E^x(\lambda(s;x),\infty) = E^x(\lambda(s;x),\infty),
 \end{align}
 which proves the first inequality of \eqref{te}.

 Now, we  assume that  $$E_\cN(e) E^x(\lambda (s;x),\infty) < E^x(\lambda (s;x),\infty) .$$
 This implies that
 \begin{align}\label{EcNe1}
  \tau( E_\cN(e) E^x(\lambda (s;x),\infty))< \tau(E^x(\lambda (s;x),\infty) ).
  \end{align}
  Since $\tau(E^x(\lambda (s;x), \infty ))\le s$,
  it follows that $\lambda (t;E_\cN(e) E^x(\lambda (s;x),\infty)  ) =0$ when $t>s$.
  Hence,
   \begin{align*}
   \int_0^s  \lambda (t;E_\cN(e) E^x(\lambda (s;x),\infty) ) dt =\tau( E_\cN(e) E^x(\lambda (s;x),\infty))&\stackrel{\eqref{EcNe1}}{<} \tau(E^x(\lambda (s;x),\infty) )\\
   &~=  \int_0^s  \lambda (t;  E^x(\lambda (s;x),\infty) )dt . \end{align*}
   In particular,
    \begin{align}\label{1-}
    \int_0^s 1- \lambda (t;E_\cN(e) E^x(\lambda (s;x),\infty) ) dt >0   .
    \end{align}
%(i.e., the projection $E_\cN(e) E^x(\lambda (s;x),\infty)$ is a proper subprojection of $ E^x(\lambda (s;x),\infty)$)\footnote{Since $\cN$ is abelian, it follows  $E_\cN(e) E^x(\lambda (s;x),\infty) \le  E^x(\lambda (s;x),\infty)$. If $E_\cN(e) E^x(\lambda (s;x),\infty) =  E^x(\lambda (s;x),\infty)$, then $E^{x}(\lambda (s;x),\infty) $ is  a subprojection of $ E_\cN(e)$. Hence, we may assume that  $E_\cN(e) E^x(\lambda (s;x),\infty) < E^x(\lambda (s;x),\infty)$.}.
Assume that $\tau(E_\cN(e) E^x(\lambda (s;x),\infty))=a_1 \ge  0$ and $\tau(E_\cN(e) E^x(-\infty, \lambda (s;x)] )=a_2  \ge  0 $.
%We assert that  $E_\cN(e) E^x(\lambda (s;x),\infty)\ne 0$.
%Otherwise,  $E_\cN (e) \le  E^x(-\infty,\lambda (s;x)]  $.
%Therefore,
%$$ \tau(xe) =\tau( xE^x(-\infty,\lambda (s;x)] e ) \le \int_0^s \lambda (s;xE^x(-\infty,\lambda (s;x)] ) dt \stackrel{\eqref{2.3}}{<} \int_0^s \lambda (t;x)dt ,$$
%which is a contradiction with the assumption.
%Hence, $\tau(E_\cN(e) E^x(\lambda (s;x),\infty))=a_1 > 0$.
Observe that $a_1+a_2 =s=\tau(E_\cN(e))$.
Moreover,  by \eqref{disx}, for every $0 <t< \tau(E^x( \lambda (s;x) ,\infty ))$, we have  $\lambda (t;x) > \lambda (s;x)$. Hence,
\begin{align*}
&~\quad  \int_0^s  \lambda (t;x)dt \\
 &~ =  \int_0^s  \lambda (t;x)  \lambda (t;E_\cN(e) E^x(\lambda (s;x),\infty)  )dt+  \int_0^s  \lambda (t;x)  (1- \lambda (t; E_\cN(e) E^x(\lambda (s;x),\infty)  )dt \\
 &~ \stackrel{\eqref{1-}}{>}   \int_0^s  \lambda (t;x)  \lambda (t;E_\cN(e) E^x(\lambda (s;x),\infty)  )dt+   \lambda (s;x)  \int_0^s  (1- \lambda (t; E_\cN(e) E^x(\lambda (s;x),\infty)  )dt.
 \end{align*}
Recall that  $\int_0^s  \lambda (t;E_\cN(e) E^x(\lambda (s;x),\infty) ) dt =\tau( E_\cN(e) E^x(\lambda (s;x),\infty))$. The above inequality implies that
\begin{align*}  \int_0^s  \lambda (t;x)dt & ~>   \int_0^1  \lambda (t;x)  \lambda (t;E_\cN(e) E^x(\lambda (s;x),\infty)  )dt+     \lambda (s;x) \Big(s - \tau( E_\cN(e) E^x(\lambda (s;x),\infty)  ) \Big) \\
    &~ =  \int_0^1  \lambda \Big(t;x  \Big)  \lambda \Big(t;E_\cN(e) E^x(\lambda (s;x),\infty)  \Big)dt+     \lambda (s;x)   a_2 .
    \end{align*}
Since $ \lambda (s;x) a_2 = \tau( \lambda (s;x)E_\cN(e) E^x(-\infty, \lambda (s;x)] ) \ge  \tau( E_\cN(e) E^x(-\infty, \lambda (s;x)] x)$, it follows that
\begin{align*}  \int_0^s  \lambda (t;x)dt  &\stackrel{\eqref{2.4}}{> }    \tau \Big( xE^x(\lambda (s;x),\infty)  E_\cN(e)   \Big)   +         \tau\Big(  x E^{x}(-\infty,\lambda (s;x)]E_\cN(e)     \Big)=  \tau(xE_\cN(e)   ),
 \end{align*}
which is a contradiction with \eqref{ecnex}.
Hence, the equality $E_\cN(e)E^x (\lambda (s;x),\infty)= E^x (\lambda (s;x),\infty)$ holds, and therefore, by \eqref{ENE>LAMBDA}, we have $E_\cN(e) \ge E^x (\lambda (s;x),\infty)$.
\end{proof}

\begin{lem}\label{reduction to commutative} Let $x\in L_1(\mathcal{M},\tau)_h$. Let $0<s< \tau(\mathbf 1)=1$ and let $a $ be in the unit ball of $ \mathcal{M}_+ $  such that $\tau(a)=s$  and $\tau(xa)=\int_0^s\lambda(t;x)dt.$
If $\lambda(x)$ is not a constant in any left neighborhood of $s,$ then $a=E^ x(\lambda(s;x),\infty).$
\end{lem}
\begin{proof} %Without loss  of generality, we may  assume $\tau(\mathbf 1)=1$.
Let $\mathcal{N}$ be  the  commutative von Neumann subalgebra of $\mathcal{M}$ generated by the spectral projections of $x$. Clearly, the restriction of $\tau$ to $\mathcal{N}$ is finite.
 There exists a conditional   expectation  $E_\cN$ from $L_1(\mathcal{M},\tau)$ to $ L_1(\mathcal{N},\tau)$~\cite[Proposition 2.1]{DDPS}.
In particular, for any $z\in L_1(\cM,\tau)$, we have $ E_\cN (z) \prec\prec z  $ (see e.g. \cite[Proposition 2.1 (g)]{DDPS})
and,
$$ E_\cN (xz)  = x E_\cN(z ).  $$
Moreover, for every $z\in \cM$ and $y\in L_1(\cN,\tau)$, we have
\begin{align}\label{exp}
\tau(yz)=\tau(E _\cN (yz))=\tau(y E_\cN( z) ).
\end{align}
%for any $x\in L_1(\cM,\tau)$, $E _\cN (x)$ is the unique element in $L_1(\cN,\tau)$ such that
%\begin{align}\label{exp}
%\tau(xy)=\tau(E _\cN (x)y)
%\end{align}
%for all $y\in \cN$

%Since $T(a)$ is a  restriction of $a$ %and $\lambda(e)=\chi_{[0,\tau(e))}$ (see \cite[Chapter III. Example 5.10 (i)]{DPS} and also \cite[formula~$(1)$]{ACS}),
Since $a$ is positive, it follows that $\lambda(a)=\mu(a)$ (see Section \ref{Preliminaries} or \cite{Hiai}) and, therefore,
%it may be concluded that
$$\int_0^r\lambda(t;E_\cN (a))dt\leqslant\int_0^r\lambda(t;a )dt%=\int_0^r\chi_{[0,\tau(e))}(t)dt
$$
for all $r\in(0,1)$ and, %it follows from \eqref{spectral scale}  that
 $$\tau(E_\cN  (a ))\stackrel{\eqref{exp}}{=} \tau(a) .$$
 %$$\int_0^{1}\lambda(t;Te)dt=\tau(T(e))\stackrel{\eqref{exp}}{=} \tau(e)=\int_0^{\tau(\mathbf 1)}\lambda(t;e)dt=\int_0^{1}\chi_{[0,\tau(e))}(t)dt.$$
%This gives $E _\cN (a )\prec a,$ which implies that
Moreover, since $E_\cN$ is a contraction on $\cM$ and $\left\|a\right\|_\infty \le 1$, it follows that $\lambda(E _\cN (a ))\leqslant1$ \cite[Proposition 2.1 (g)]{DDPS}.
We set  %$s_1:= \tau(E^{x}(\lambda (s;x), \infty))$, $s_2:= \tau(E^{x}[\lambda (s;x), \infty))$ and
$y:=(x-\lambda(s;x) )_+ .$
Note that
$$x\leqslant\lambda(s;x)+y .$$
Therefore,
\begin{equation}\label{trace}
\tau(xa )\leqslant\tau(\lambda(s;x)a)+\tau(ya)=s\cdot \lambda(s;x) +\tau(ya).
\end{equation}
Since $\lambda(x)$ is a decreasing function, we have
  \begin{align}\label{yfinite}\lambda(t;y)=\left\{
                 \begin{array}{ll}
                   \lambda(t;x)-\lambda(s;x), & \hbox{if}\ 0<t<s; \\
                   0, & \hbox{if}\ s\leqslant t\leqslant1.
                 \end{array}
               \right.
               \end{align}
It follows from $\tau(a)=s$  and $\tau(xa)=\int_0^s\lambda(t;x)dt$ that
\begin{align*}
  \int_0^s\lambda(t;y)dt &\stackrel{\eqref{yfinite}}{=} \int_0^s\lambda(t;x)dt-\int_0^s\lambda( s  ;x) dt  \\
  &~= \tau(xa)- \lambda(s;x) \tau( a )  \\
  &\stackrel{\eqref{trace}}{\leqslant} \tau(ya) \\
  &\stackrel{\eqref{exp}}{=}  \tau(yE _\cN (a)) \\
  &\stackrel{\eqref{2.4}}{\leqslant} \int_0^{1}\lambda(t;y)\lambda(t;E_\cN (a ))dt \\
  &\stackrel{\eqref{yfinite}}{=} \int_0^{s}\lambda(t;y)\lambda(t;E _\cN(a))dt.
\end{align*}
%Let $y= x + (\varepsilon -\lambda (s;x) )$,
%Then, $\lambda (s^-;x)\ge \varepsilon$.
%\begin{align*}
%  \int_0^s\lambda(t;y)dt &~= %\int_0^s\lambda(t;x)dt-\int_0^s\lambda(s;x) dt + (s-s_1)e pe \\
%  &~= \tau(xe)-s(\lambda(s;x)) + (s-s_1) epe \\
%  &~\leqslant \tau(ye) \\
%  &\stackrel{\eqref{exp}}{=}  \tau(yT(e)) \\
%  &~\leqslant \int_0^{1}\lambda(t;y)\lambda(t;Te)dt \\
%  &~= \int_0^{s}\lambda(t;y)\lambda(t;Te)dt.
%\end{align*}
Thus,
$$\int_0^{s}\lambda(t;y)(1-\lambda(t;E_\cN  (a)))dt\leqslant0.$$
Since $y$ is positive and $\lambda(E_\cN(a))\leqslant1,$ we conclude that $\lambda(t;y)(1-\lambda(t;E_\cN(a)))\geqslant0.$
Hence,  $\lambda(t;y)(1-\lambda(t;E_\cN(a)))=0$ for all $t\in(0,s).$
Recall that $\lambda(x)$ is not a constant in any left neighborhood of $s$. We obtain that  $\lambda (y)>0$ on $(0,s)$.
Recall that $E_\cN(a)\ge 0$ with $\tau(E_\cN(a) )=\tau(a) =s$.
 We obtain that  $\lambda (E_\cN(a))=1$ on $(0,s]$ and $\lambda (E_\cN(a))=0$ on $[s,1)$.
This implies that
$E_\cN(a)$ is a projection in $\cN$.
%Since $\lambda(x)$ is not a constant in any left  neighborhood of $s,$ we obtain that $\lambda(t;x)>\lambda(s;x)$ for all $t\in(0,s).$ Consequently, $\lambda (y)>0$ and   $\lambda(Te)=1$ on $(0,s),$ i.e. $\lambda(Te)=\chi_{[0,s) }$ and $Te\in P(\mathcal{N}).$
%Since $\tau(e)=\tau(e\cdot e)=\tau(eT(e))=\tau(eT(e)e)$,
%it follows that  $\tau(e(1-T(e))e)=0.$ Therefore %$|(1-Te)^{\frac12}e|^2=e(1-Te)e=0.$
%Consequently, $(1-Te)^{\frac12}e=0.$
%This shows that $(1-Te)e=0,$ i.e. $e=T(e)\cdot e,$
%which implies that $e \le T(e) e e T(e) \le T(e)^2$.
%By the monotonicity of function $t\mapsto t^{1/2}$, we have  %$e\leqslant T(e).$
% From $\tau(Te)=\tau(e)$, we obtain  that $\tau(Te-e)=0,$ i.e.,  $Te-e=0$ and hence $Te=e\in P(\mathcal{N}).$
%Recalling that  $\mathcal{N}$ is a commutative von Neumann subalgebra %of $\mathcal{M}$ generated by    all spectral  projections  of $x$,  the assertion follows immediately (see e.g. \cite[Chapter III, Theorem 8.5 and Lemma 7.10 (i)]{DPS}).
Hence, $E_\cN(a)= E_\cN(a) E_\cN(a) =E_\cN(a\cdot E_\cN(a))$ and $E_\cN(a( {\bf 1}-E_\cN(a)))=0$.
It follows that $$\tau(a^{1/2} ({\bf 1}-E_\cN(a)) a^{1/2} ) =\tau(a ({\bf 1}-E_\cN (a))  )=\tau(E_\cN(a( {\bf 1} -E_\cN(a)))  ) =0  .$$
Therefore, $a^{1/2}({\bf 1} -E_\cN(a))a^{1/2} =0$ and $a^{1/2} =E_\cN (a) a^{1/2} $.
By the assumption that  $a\le 1$, we have
$$ E_\cN(a)= E_\cN(a)^2   \ge   E_\cN(a)a E_\cN(a)   =   a . $$
Recall that $\tau(E_\cN(a))=\tau(a)$.
Hence, $\tau(E_\cN(a)-a )=0$.
Due to the faithfulness of the trace $\tau$, we obtain that $a = E_\cN(a )  \in \cP(\cN)$.
Since  $\mathcal{N}$ is a commutative von Neumann subalgebra of $\mathcal{M}$ generated by    all spectral  projections  of $x$,
 it follows that $a= E^{x}\{B\}$ for some Borel set $B\subset [0,1]$.
 By \cite[Proposition 1]{P} and the assumption that $\tau(ax)=\int_0^s \lambda (t;x)dt$, we have
$$\int_0^s \lambda (t;x)dt =\tau(a x )=\int_0^1 \chi_{B}(\lambda(t;x))dt .$$ Moreover, since  $\lambda$ is decreasing and is non-constant  in any left neighborhood of $s$, it follows  that $B=(\lambda(s;x),\infty)$. %(see e.g. \cite[Chapter III, Theorem 8.5 and Lemma 7.10 (i)]{DPS}).
\end{proof}

\begin{rmk}\label{posi}
Let $x\in L_1(\cM,\tau)_+$.
By \eqref{2.4},  for any $a$ in the unit ball of $\cM _+$ with $\tau(a)< s$, we have
$$\tau(xa) \le \int_0^1 \lambda (t;x) \lambda (t;a)dt.$$
Since $\lambda (x)$ is a non-negative decreasing function, $0\le \lambda (a)\le 1 $ and $\tau(a)=\int_0^1 \lambda(t;a)dt <s$, it follows that
\begin{align}\label{<<<}
\tau(xa) \le \int_0^1 \lambda (t;x) \lambda (t;a)dt <  \int_0^s  \lambda (t;x)dt  .
\end{align}
Then, by \eqref{Hineq}, we obtain that
$$\int_0^s   \lambda (t;x)dt =\sup \{\tau(xa);a\in \cM, 0\le a\le 1, \tau(a)\le s\}.$$
Lemma \ref{reduction to commutative} together with \eqref{<<<} implies that  if  $a $ is  in the unit ball of $ \mathcal{M}_+ $  such that $\tau(a)\le  s$  and $\tau(xa)=\int_0^s\lambda(t;x)dt,$
and
$\lambda(x)$ is not a constant in any left neighborhood of $s,$ then $a=E^ x(\lambda(s;x),\infty).$
\end{rmk}

\begin{cor}\label{corcons} Let $x\in L_1(\mathcal{M},\tau)_h $. Let $0<s< \tau(\mathbf 1)=1 $ and let $e \in  \cP(\mathcal{M})$ be  such that $\tau(e )=s$ and $\tau(xe)=\int_0^s\lambda(t;x)dt.$
Then,
$$E^{x }(\lambda (s;x   ),\infty ) \le  e \le E^{x }[\lambda (s;x ),\infty ).$$
\end{cor}
\begin{proof}
By Lemma \ref{reduction to commutative}, it suffices to prove the case when $ \lambda (x)$ is a constant on a left neighbourhood of $s$.
Denote by $\cN$ the von Neumann algebra generated by all spectral projections of $x $.
Let $E_\cN $ be a conditional expectation from $L_1(\cM,\tau)$ onto $ L_1(\cN,\tau)$.
Let $\lambda := \lambda (s;x)$ and  $x_1: =(x-\lambda )E^x (\lambda,\infty)$.
Recall that $E^x[ \lambda,\infty) \ge E _\cN (e)\ge E^x(\lambda,\infty)$ (see \eqref{te}).
In particular, $E_\cN(e)E^x(\lambda,\infty) =E^x(\lambda,\infty) $.
Observing that $E^x(\lambda,\infty)$ is the support of $x_1$, we have
\begin{align*}
\tau(x_1  E^x (\lambda,\infty) e E^x (\lambda,\infty)) &~=\tau(x_1 e  )=\tau( x_1  E_\cN(  e) )\\
&~=\tau( x_1 E^x (\lambda,\infty) E_\cN( e)) \\
&~=\tau( x_1 E^x (\lambda,\infty)  ) \\
&~=\tau( (x-\lambda ) E^x (\lambda,\infty)  ) \\
&~=\tau( x E^x (\lambda,\infty)  ) -\tau( \lambda E^x (\lambda,\infty)  ) \\
&~=\int_0^{\tau(E^x (\lambda,\infty))} (\lambda (t;x E^x (\lambda,\infty) ) -\lambda )dt\\
&\stackrel{\eqref{2.2}}{=}\int_0^{\tau(E^x (\lambda,\infty))} (\lambda (t;x) -\lambda )dt\\
&~=\int_0^{\tau(E^{x_1} (0,\infty))} \lambda (t;x_1)dt.
\end{align*}
Note that $0\le E^x(\lambda,\infty   )eE^x(\lambda,\infty )\le \textbf{1}$ and $\tau(E^x(\lambda,\infty   )eE^x(\lambda,\infty )) \le \tau(E^x(\lambda,\infty   ) ) $.
Since $x_1\ge 0$, it follows from  Remark \ref{posi} that
$$E^x (\lambda,\infty) e E^x (\lambda,\infty) =E^{x_1} (0,\infty)=E^{x } (\lambda ,\infty)  .$$
That is,
$e \ge E^x (\lambda,\infty) $.
Let $e_1: = e -E^{x}(\lambda,\infty)\in \cP(\cM)$.
We have $$\tau(x e) =\tau(x ( E^{x}(\lambda,\infty)  + E^{x}(-\infty,\lambda]) e  )=\tau(x E^{x}(\lambda,\infty)  +x E^{x}(-\infty,\lambda]e_1).$$
Hence, by the  assumption that $\lambda=\lambda(s;x)$, we obtain that
\begin{align*}
\tau(x E^{x}(-\infty,\lambda]e_1)& = \int_{0}^s \lambda (t;x)dt -\int_0^{\tau(E^{x}(\lambda,\infty))} \lambda (t;x)dt\\
&=\int_{\tau(E^{x}(\lambda,\infty))}^s \lambda (t;x)dt  =\tau(\lambda e_1) .
\end{align*}
We have
$$ \tau(e_1 (\lambda  - x E^{x}(-\infty,\lambda ])e_1) =0.$$
Therefore, $ e_1 (\lambda - x E^{x}(-\infty,\lambda])e_1=0  $.
Since $\lambda - x E^{x}(-\infty,\lambda]\ge 0$, it follows that
$(\lambda - x E^{x}(-\infty,\lambda])^{1/2}e_1=0$.
Hence,  $$ \left(\int_{t <\lambda  } (\lambda -t) dE_t^x + \int_{t > \lambda  } \lambda  dE_t^x \right) e_1 =(\lambda - x E^{x}(-\infty,\lambda])e_1  = 0$$
and
\begin{align*}
E^x  \Big( (-\infty,\lambda)\cup (\lambda ,\infty)  \Big) \cdot e_1 = \left(\int_{t <\lambda  } \frac{1}{\lambda -t} dE_t^x + \int_{t > \lambda  } \frac1\lambda    dE_t^x \right) \left(\int_{t <\lambda  } (\lambda -t) dE_t^x + \int_{t > \lambda  } \lambda  dE_t^x \right) e_1 = 0.  %\\ &=E^x(-\infty,\lambda)\cup (\lambda ,\infty) e_1  .
\end{align*}
%This implies that $e _1\perp E^x(-\infty,\lambda)\cup (\lambda ,\infty)$.
This  implies that  $ e_1\le E^x\{\lambda (s;x)\} $, which completes the proof.
\end{proof}

The following proposition is similar to a well-known property of rearrangements of functions, see \cite[property $9^0$, p. 65]{KPS} and \cite[Theorem 3.5]{Hiai}.

\begin{prop}\label{noncomm KPS lemma} Let $x,\ x_1,\ x_2\in L_1(\mathcal{M},\tau)_h$ be such that $x=(x_1+x_2)/2$ and $\lambda(x_1)=\lambda(x_2)=\lambda(x).$
Then, $x=x_1=x_2$.
\end{prop}

\begin{proof}
Fix $\theta \in (\lambda(1^-;x) , \lambda (0,x) )  $.
Define $s$ by setting
$$s := \min \{ 0\le v\le 1 :\lambda (v;x) \le  \theta \}. $$
If $s>0$,  then $$ \lambda(s;x) \le \theta <   \lambda(s-\varepsilon;x), ~\forall \varepsilon\in (0,s). $$
%{\color{red} and when $s=0$, we have $\theta =  \lambda(0;x). $}
This means that $\lambda(x)$ is not constant in any left neighborhood of $s$ whenever $s>0$.

Fix a projection $e=E^{x_1+x_2}(\lambda(s; x_1+x_2),\infty)=E^{2x}(\lambda(s; 2x ),\infty)=E^{x}(\lambda(s; x ),\infty)\in\mathcal{M}.$ Clearly, $\tau(e)=s$.  %(see e.g. \cite[Chapter III, Lemma 8.4]{DPS}).
By \eqref{Hineq}, we have
\begin{align*}
 \tau(e(x_1+x_2)) &~= \int_0^{s}\lambda(t;x_1+x_2)dt\\
&~= \int_0^{s}2 \lambda(t;x)dt\\
 &~= \int_0^{s}\lambda(t;x_1)dt+\int_0^{s}\lambda(t;x_2)dt\\
 &\stackrel{\eqref{Hineq}}{\geq}  \tau(ex_1)+\tau(ex_2)\\
 &~= \tau(e(x_1+x_2)).
\end{align*}
Hence,  $\int_0^{s}\lambda(t;x_i)dt =\tau(ex_i),\; i=1,2$.
 From Lemma \ref{reduction to commutative}, we obtain that $e=E^{x}(\lambda(s; x ),\infty) =E^{x_1}(\lambda(s;x_1),\infty)=E^{x_2}(\lambda(s;x_2),\infty).$
Hence,
we have
\begin{align}\label{x012}
E^x(\theta,\infty) = E^x(\lambda(s;x),\infty)& = E^{x_1}(\lambda (s;x_1),\infty)= E^{x_2}(\lambda (s;x_2),\infty)\nonumber\\
&= E^{x_1}(\theta,\infty)=E^ {x_2}(\theta,\infty).
\end{align}
Note that (for convenience, we denote $E^x [\infty,\infty )=E^{x_1}[\infty,\infty )=E^{x_2} [\infty,\infty )=0$)
\begin{align*} E^x [\lambda(0;x),\infty ) & ~= {\bf 1}- E^x (-\infty, \lambda(0;x) ) ={\bf 1} -\lim_{\theta \uparrow  \lambda(0;x)^-}E^x (-\infty, \theta  ] \\
&\stackrel{\eqref{x012}}{=} {\bf 1}-\lim_{\theta \uparrow  \lambda(0;x)^-}E^{x_1} (-\infty, \theta  ]= {\bf 1}- E^{x_1} (-\infty, \lambda(0;x_1) ) =E^{x_1} [\lambda(0;x_1 ),\infty ) \\
&\stackrel{\eqref{x012}}{=}  {\bf 1}-\lim_{\theta \uparrow  \lambda(0;x)^-}E^{x_2} (-\infty, \theta  ]= {\bf 1} - E^{x_2} (-\infty, \lambda(0;x_2 ) ) =E^{x_2} [\lambda(0;x_2),\infty ),
\end{align*}
which together with \eqref{x012} implies that
$$ x E^x (\lambda(1^-;x ),\infty ) = x_1 E^{x_1} (\lambda(1^- ;x_1),\infty )= x_2 E^{x_2} (\lambda(1^-;x_2),\infty )  . $$
We note by $l(z)$ the left support of $z\in S(\cM,\tau)$.
Since $\lambda(1^-;x)=\lambda(1^-;x_1)=\lambda(1^-;x_2)$ and $E^{x}\{\lambda(1^-;x)\} = {\bf 1}-   l  ((x-\lambda(1^-;x))_+ ) =  {\bf 1}-  l  ((x_1-\lambda(1^-;x_1))_+ )= E^{x_1}\{\lambda(1^-;x_1)\} = {\bf 1}-  l  ((x_2-\lambda(1^-;x_2))_+ )=  E^{x_2}\{\lambda(1^-;x_2)\}$,
 it follows that $$x=x_1=x_2.$$
\end{proof}

The following lemma provides the proof of the implication ``$\Leftarrow$'' in Theorem \ref{RyffNonCom}.
\begin{lem}\label{lemma:ex}
Let $y\in L_1(\cM,\tau)$.
Let  $x,x_1,x_2\in \Omega (y)$ with $x=\frac{x_1+x_2}{2}$.
If $x$ satisfies that  for   every $t\in (0,1)$, one of the followings holds:
 \begin{enumerate}
  \item[(1).]   $\lambda(t;y)=\lambda(t;x)$;
  \item [(2).]  $\lambda(t;y) \ne \lambda(t;x)$ with $ E^x \{\lambda (t;x )  \} $ is an atom and $$\int_{ \{s : \lambda (s;x)=\lambda (t;x )\}}  \lambda (s;y)ds    =  \lambda(t ;x)     \tau(E^x (\{\lambda (t;x)\}))  ,   $$
\end{enumerate}
 then $x_1=x_2=x$.
\end{lem}
\begin{proof}
%Consider the atomless von Neumann algebra $\cM\otimes L_\infty(0,1)$.
By the definition of $x$ and assumption (2) above, for every $t$   such that $\lambda (t;x) =\lambda(t; y)$,
we have $\int_0^t \lambda (s;x)ds = \int_0^t \lambda (s;y)ds$.
For any $t$ such that $\lambda(t;y) \ne \lambda (t;x)$, we denote $[t_1,t_2) =\{s : \lambda (s;x) =\lambda (t;x )\}$, $t_1<t_2$.
In particular, we have
$$\int_0^{t_1}\lambda (s;y)ds =\int_0^{t_1}\lambda (s;x)ds $$
and
$$\int_0^{t_2}\lambda (s;y)ds =\int_0^{t_2 }\lambda (s;x)ds .$$
Since   $x_1,x_2\in \Omega(y)$ and
$$  \int_0^{t_1} \lambda (s;2x)ds = \int_0^{t_1} \lambda (s;x_1+x_2)ds\stackrel{\eqref{tri}}{\le}  \int_0^{t_1} \lambda (s;x_1)ds+\int_0^{t_1} \lambda (s;x_2)ds \le 2 \int_0^{t_1}\lambda (s;y)ds, $$
it follows that
\begin{align}\label{***}  \int_0^{t_1} \lambda (s;x)ds =   \int_0^{t_1} \lambda (s;x_1)ds= \int_0^{t_1} \lambda (s;x_2)ds = \int_0^{t_1}\lambda (s;y)ds.  \end{align}
The same argument with $t_1$ replaced with $t_2$ yields that
\begin{align}\label{****}  \int_0^{t_2} \lambda (s;x)ds =   \int_0^{t_2} \lambda (s;x_1)ds= \int_0^{t_2} \lambda (s;x_2)ds = \int_0^{t_2} \lambda (s;y)ds.  \end{align}
Let $e_1 := E^{x}(\lambda (t_1 ;x),\infty)$ and  $e_2:= E^{x}[ \lambda ( t_1 ;x),\infty)  $.
In particular, $e_2-e_1 =E^x\{\lambda (t_1;x)\}$.
Observe that $\tau(e_1)=t_1$ and $\tau(e_2)=t_2$ (due to the assumption that $[t_1,t_2)=\{s : \lambda (s;x) =\lambda (t;x )\} $).
%Assume that $t$ is such that $\lambda (x) =\lambda (t ;x) $ on $[t_1,t_2)$.
% We have
% $$ \int_0^{t_1} \lambda (s;x)ds=\int_0^{t_1} \lambda (s;x_1)ds=\int_0^{t_1} \lambda (s;x_2)ds=\int_0^{t_1} \lambda (s;y)ds.$$
Note that, by \eqref{Hineq} and the definition of spectral scales $\lambda (x)$, we have
\begin{align*}
 2 \int_0^{t_1} \lambda (s;x)ds & = \tau(2x e_1) =\tau((x_1+x_2)e_1) =\tau(x_1e_1)+\tau(x_2e_1 )\\
  &\stackrel{\eqref{Hineq}}{\le}  \int_0^{t_1} \lambda (s;x_1)ds+\int_0^{t_1} \lambda (s;x_2)ds =2\int_0^{t_1}  \lambda (s;x)ds .
 \end{align*}
 We obtain that $\tau(x_1e_1)=\int_0^{t_1} \lambda (s;x_1)ds
=\int_0^{t_1} \lambda (s;x_2)ds = \tau(x_2e_1 )   $.
By Corollary \ref{corcons},
 we have
 $$ E^{x_1}(\lambda (t_1;x_1),\infty ) \le  e_1  \le E^{x_1}[\lambda (t_1;x_1),\infty )$$ and
 $$ E^{x_2}(\lambda (t_1 ;x_2),\infty ) \le  e_1  \le E^{x_2}[\lambda (t_1;x_2),\infty ) .$$
 Similar argument   with $t_1$ replaced with $t_2$ yields that
  $$ E^{x_1}(\lambda (t_2;x_1),\infty ) \le  e_2 \le E^{x_1}[\lambda (t_2;x_1),\infty )$$
  and
   $$ E^{x_2}(\lambda (t_2;x_2),\infty ) \le  e_2  \le E^{x_2}[\lambda (t_2 ;x_2),\infty ) .$$
 In particular, $e_1=E^{x_1}(\lambda (t_1;x_1),\infty )  + q$ for some subprojection $q$ of $E^{x_1}\{\lambda (t_1;x_1)\}$.
 Since $q$ commutes with  $E^{x_1}\{\lambda (t_1;x_1)\}$, it follows that $q $ commutes with any spectral projection of $x_1$. Hence, $e_1$ commutes with $x_1$.
   The same argument implies that both
   $e_1$ and $e_2$ commute with $x_1$ and with  $x_2$. %(see e.g. \cite{DDP} and  \cite[Chapter III, Lemma 7.10]{DPS}).
Moreover,
 the atom $  e:= e_2-e_1\in \cP(\cM)$ satisfies that
 $$ e   \le E^{x_1} [\lambda (t_2;x_1), \lambda (t_1;x_1)) \mbox{ and } e\le  E^{x_2} [\lambda (t_2;x_2),\lambda (t_1;x_2) ) . $$
%%%%%%We have
%%%%%%%%%$$x_1 =x_1 (1-e) +x_1e \mbox{ and } x_2 =x_2 (1-e) +x_2e .$$
By the spectral theorem, % we get  $x_1e=x_2e = \lambda (t;x)e$.
%Hence, take all $e$,
%$$\lambda (x - \sum \lambda _n e_n)  =\lambda (x_1 - \sum \lambda _n e_n) =\lambda (x_2 - \sum \lambda _n e_n ) $$
 %Note that
 $ \lambda (x_1e_1) =\lambda (x_1)$  and $ \lambda (x_2e_1) =\lambda (x_2)$ on $(0,t_1)$ (see  \eqref{2.2}).
 On the other hand,
 $$\lambda (t; x_1e ) = \lambda (t; x_1e_2 e ) \stackrel{\eqref{2.3}}{=} \lambda (t+t_1; x_1e_2) \stackrel{\eqref{2.2}}{=}\lambda (t+t_1; x_1)$$
 and $$\lambda (t; x_2e ) = \lambda (t; x_2e_2 e ) \stackrel{\eqref{2.3}}{=}  \lambda (t+t_1; x_2e_2 ) \stackrel{\eqref{2.2}}{=} \lambda (t+t_1; x_2) $$
 for all $t\in [0,t_2-t_1)$.
 Since $e $ is an atom, it follows that $\lambda _1 :=\lambda (t; x_1e ) = \lambda (t+t_1; x_1)$
 and $\lambda _2 :=\lambda (t; x_2e ) = \lambda (t+t_1; x_2)$ for every $t\in [0,t_2-t_1)$.
Indeed, combining \eqref{***} and \eqref{****}, we have
 $$  \int_{t_1}^{t_2} \lambda (s;x)ds =    \int_{t_1}^{t_2} \lambda (s;x_1)ds =\lambda _1 (t_2-t_1)=  \int_{t_1}^{t_2} \lambda (s;x_2)ds =  \lambda _2  (t_2-t_1)= \int_{t_1}^{t_2} \lambda (s;y)ds. $$
Hence,  $\lambda_1 =\lambda _2 = \lambda (t_1 ;x) $.
% $$\int_0^{t_2}  \lambda (s;x)ds= \int_0^{t_2}  \lambda (s;x_1)ds  =\int_0^{t_1} \lambda (s;x_1)  + \int_{0 }^{t_2-t_1} \lambda ( x_1  e )=\int_0^{t_1} \lambda (s;x_1)  + \int_{0 }^{t_2-t_1}
 %  \lambda_1 e   $$
 %This implies that $ \lambda _1=\lambda _2 = \lambda (t_1;x)$.
That is,  $\lambda (x_1)=\lambda (x_2) =\lambda (x) $ on $[t_1,t_2)$.
Since $t$ is arbitrary,
 it follows that  $\lambda (x_1) =\lambda (x_2)=\lambda (x)$.
  By Proposition \ref{noncomm KPS lemma}, we obtain that $x_1=x_2=x$.
% by $2x=x_1+x_2$, there exist $e_1,e_2$ with $\tau(e_1)=\tau(e_2)$ such that
%$$2\int_0^{t_1}  \lambda (s;x)ds = \tau(x_1e_1)+ \tau(x_2e_2) =\tau(  (x_1+x_2)e) =\tau(  x_1e +x_2e )   .$$
%Hence, $\tau(x_1e_1) =\int_0^t \lambda (s;x_1)ds$.
%By Corollary \ref{corcons}, we have $ E^{x_1}(\lambda (t;x_1),\infty ) \le  e \le E^{x_1}[\lambda (t;x_1),\infty )$.
%Hence, $ E^{x} (\lambda (t;x),\infty) =e _1'  =e_2$.
%The same for the atom, we get
%$$E^{x}(\lambda (t+\tau(e)),\infty)  = e_1'=e_2'.  $$
%Hence, $e_1'-e_1 =e_2'-e_2$ is the atom in $\cM$.
%By
%$$\int_0^t \lambda (s;x)ds = \int_0^t \lambda (s;x_1) = \lambda _0^t %\lambda (s;x_2)ds = \lambda _0^t \lambda (s;y)ds $$
%and
%$$\int_{t'}^t \lambda (s;x)ds  = \int _{t'}^t \lambda (s;y)ds$$
\end{proof}

%By \cite[Lemma 4.4]{HSZ}

%\begin{lem}\label{lemm8}
%Let $g_1,g_2\in L_1(0,1)$.
%If
%$\lambda( g_1 +g_2) =\lambda(g_1) +\lambda(g_2)=2\lambda (g_1) =2\lambda (g_2) $,
%then $g_1 =g_2$ a.e.
%\end{lem}
%\begin{proof}
%Let $x_1 := (g_1)_+  $, $(x_2 ):= (g_2)_+$ and $x:=(g_1+g_2)_+$.
%We have
%$$\lambda (x)=  \lambda (g_1+g_2)_+ = (\lambda(g_1)+\lambda (g_2)) _+ =2\lambda (g_1)_+ =2\lambda (g_2)_+ = 2 \lambda(x_1)=2 \lambda (x_2)  . $$
%Moreover, we have  $x_1 +x_2\ge 0$, $x_1+x_2 \ge g_1 +g_2$,
%$$  \mu(x_1 +x_2)  \ge \mu((g_1+g_2)_+)  =\lambda ((g_1+g_2)_+ ) = \lambda (x) = \mu(x), $$
%that is, $$ \mu(x_1+x_2 )\ge \mu(x). $$
%By $\mu(x_1) = \lambda ((g_1)_+) = \lambda ((g_2)_+)  =\mu(x_2)$,
%we have
%$$\mu(x) = \mu((g_1+g_2)_+) = \lambda (g_1+g_2)_+  =2 \lambda (x_1) =2\lambda (x_2) =2\mu(x_1)=2\mu(x_2).  $$
%Thus,
%$$\mu(x_1) +\mu(x_2)  \succ\succ \mu(x_1+x_2 )\ge \mu(x) =\mu(x_1) +\mu(x_2). $$
%Hence, $ \mu(x_1+x_2 )=\mu(x_1) +\mu(x_2)$ and so,
%$ x_1= x_2  $ (see e.g. \cite[p.65]{KPS}).

%Applying the latter to
%  $(x_1)_-  =(-x_1)_+$ and $(x_2)_-  =(-x_2)_+$, we obtain that $(x_1)_- =(x_2)_-$.
%\end{proof}
\section{Proof of the main result}
Now, we prove the main result of the present paper.

\begin{proof*}[of Theorem \ref{RyffNonCom}]
``$ \Rightarrow$''.  Assume that $x\in \text{extr}(\Omega(y))$.
We assert that for   every $t\in (0,1)$, one of the followings holds:
\begin{enumerate}
  \item[(1).]   $\lambda(t;y)=\lambda(t;x)$;
  \item [(2).]  $\lambda(t;y) \ne \lambda(t;x)$ with $ E^x \{\lambda (t;x)  \} $ an  atom and $$\int_{ \{s: \lambda (s;x)=\lambda (t;x)\}}  \lambda (s;y)ds    =  \lambda(t ;x)     \tau(E^x \{\lambda (t;x)\})  .  $$
\end{enumerate}

%Assume by contradiction that $A$ not empty.
We set
$$A=\left\{s:\int_0^s\lambda (t;y)-\lambda (t; x)dt>0\right\}.$$
Since $\lambda (y),\lambda (x)\in L_1(0,1)$, it follows that the mapping $f: s\mapsto \int_0^s\lambda (t;y)-\lambda (t;x)dt$ is continuous.
Moreover, noting that  $f(0)=f(1) =0$,
  we infer that $A$ is an open set, i.e.,  $A=\cup_i(a_i,b_i)$, where
$a_i,b_i\not\in A$.  %(see \cite[8. Proposition, p. 42]{Ro} or \cite[p. 9]{Ru}).
% Thus, $\lambda (x)\chi_{(a_i,b_i)}\prec
%\lambda (y)\chi_{(a_i,b_i)}.$
By Lemma \ref{counter example 2}, $\lambda(x)$ is a step function on $(a_i,b_i )$.
%By Lemma \ref{concatenation lemma}, $x\chi_{(a_i,b_i)}\in\text{extr}(\Omega(x_{\chi_{(a_i,b_i)}})).$

We assert that   \begin{align}\label{a_idisc}
\lambda (a_i^-;x) \ne \lambda (a_i+\varepsilon ;x)
 \end{align}
 for any $\varepsilon>0$.
Note that $\lambda (a_i^-; y)\le   \lambda (a_i^-;x)$ (indeed, if $\lambda (a_i^-; y)>  \lambda (a_i^-;x)$, then, by $y\succ x$, we obtain that   $\int_0^{a_i}\lambda (t;y)-\lambda (t;x) dt >0$, which is a contradiction with $a_i\notin A$).
Assume by contradiction that \eqref{a_idisc} does not hold, that is,
 $\lambda (a_i^-; x) =\lambda (a_i+\varepsilon;x)$ for some $\varepsilon\in (0, b_i-a_i)$.
  Then, $ \lambda (a_i^-;y) \le \lambda (a_i^-; x)=\lambda (a_i+\varepsilon;x)$.
Since $\lambda(y)$ is decreasing, it follows that
$\int_{a_i}^{a_i+\varepsilon } \lambda (t;y)-\lambda (t;x)dt \le 0$.
However, since $a_i\notin A$ and $a_i +\varepsilon\in A$, it follows that  $\int_0^{a_i} \lambda (t;y)-\lambda (t;x)dt =0$ and   $\int_{a_i}^{a_i+\varepsilon } \lambda (t;y)-\lambda (t;x)dt >0$, which is a contradiction.
%
%By $x\in \Omega (y)$, we have  $ \lambda (a_i +\varepsilon;y) =\lambda (a_i+\varepsilon ;x)$, which is a contradiction with the assumption that $a_i+\varepsilon \in A$.

For a given $\varepsilon>0$, we claim  that  $\lambda (b_i-\varepsilon_1;y ) \le  \lambda (b_i-\varepsilon_1; x )$ for some $\varepsilon_1\in (0,\varepsilon)$.
Assume by contradiction that $\lambda (b_i-\varepsilon_1;y ) >   \lambda (b_i-\varepsilon_1; x )$ for all $\varepsilon_1\in (0,\varepsilon)$. Then, by $x\prec  y$, we obtain that
$\int_0^{b_i}\lambda (t;y)-\lambda (t;x) dt >0$.
 That is, $b_i\in A$, which is a contradiction with the assumption.
We assert that
\begin{align}\label{bdis}
\lambda (b_i-\varepsilon;x) > \lambda (b_i;x)
\end{align}
 for any $\varepsilon>0$.
Assume by contradiction  that $\lambda (b_i-\varepsilon; x ) = \lambda (b_i;x)$ for some $\varepsilon\in (0, b_i-a_i )$.
Then, $\lambda (b_i-\varepsilon_1;y ) \le  \lambda (b_i-\varepsilon; x ) = \lambda (b_i-\varepsilon_1; x ) = \lambda (b_i;x)  $ for some $\varepsilon_1 \in (0,\varepsilon)$.
However,
since $x\in \Omega (y)$ and $b_i \notin A$, it follows that $\int_{b_i}^{b_i+\delta} \lambda(t;y)-\lambda(t;x)dt\ge 0 $ for any $\delta>0$.
By the right-continuity of spectral scales, we obtain that
$$ \lambda (b_i ; y)\ge \lambda (b_i;x).$$
That is, $\lambda (b_i-\varepsilon_1;y ) \ge  \lambda (b_i ; y)\ge \lambda (b_i;x) =\lambda (b_i-\varepsilon_1;x) \ge \lambda (b_i-\varepsilon_1;y )$, which is a contradiction with  the assumption that $b_i -\varepsilon_1 \in A $.
% Similarly,  $\lambda (b_i^-; x) \ne \lambda (b_i+\varepsilon;x )$  for any $\varepsilon>0$.

 Fix an index $i$, we will consider the step function $\lambda (x)$ on  $(a_i,b_i)$ and construct operators $x_\pm$ such that $2x=x_+ + x_-$.
%%%%%%%%%Therefore, %$a_i$ and $b_i$ are discontinuous points of $\lambda (x)$.
%%%%%%%%%for the sake of convenience,
%%%%%%%%%replacing $x $ with  $x E^{x}( \lambda (s_2;x) ,  \lambda (s_1;x))$,  %instead of $x$.
%For the sake of convenience,
%%%%%%%%%we may assume that $A=(0,1).$

\emph{Case 1.}
Suppose that  $\lambda (x) $ takes three or more values on $(a_i,b_i)$.
Hence, there exist $a_i <s_1<s_2<s_3\le b_i $ such that $\lambda (x)|_{[s_1,s_2)}= C_1$ and  $\lambda (x)|_{[s_2,s_3)}= C_2$ with $\lambda (s_3;x ) < C_2$, $C_1>C_2$.
Let
$$u: = E^{x}\{ C_1 \}-  \frac{s_2-s_1}{s_3-s_2} \cdot E^{x}\{ C_2 \} ,$$
and $$x_\pm := x \pm \delta u . $$
It is clear that $2x =x_+ +x_-$ and
$$\tau(x)= \tau(y) =\tau(x_\pm).$$
We assert that there exists sufficiently small $\delta$ such that
$$ x _\pm \prec y. $$
Assume that $\delta>0$ satisfies that
 \begin{align}\label{delta1}
 \delta < \lambda (s_1^{ - }; x)-\lambda (s_1;x)=\lambda (s_1^-; x)-C_1,
 \end{align}
 \begin{align}\label{delta2}
 \frac{s_2-s_1}{s_3-s_2}\delta \stackrel{\eqref{bdis}}{<}\lambda (s_3^{ - };x)-\lambda (s_3;x)=C_2 -\lambda (s_3;x)
 \end{align}
and
 \begin{align}\label{delta3}
 \delta +\frac{s_2-s_1}{s_3-s_2}\delta   <   C_1-C_2  .
 \end{align}
Note  that $x$ can be represented in the form of  $$\int_{\lambda (s_1^-; x)} ^\infty t dE^{x}_t +C_1E^x\{C_1\} +C_2E^x\{C_2\} +\int_{-\infty} ^{\lambda (s_3;x)} t dE^{x}_t.$$
Recalling that $u=  E^{x}\{ C_1 \}-  \frac{s_2-s_1}{s_3-s_2} \cdot E^{x}\{ C_2 \} $ and $x_\pm = x \pm \delta u$, we have
$$x_{\pm}=\int_{\lambda (s_1^-; x)} ^\infty t dE^{x}_t +(C_1
\pm \delta)E^x\{C_1\} +(C_2 \mp\frac{s_2-s_1}{s_3-s_2}\delta )E^x\{C_2\} +\int_{-\infty} ^{\lambda (s_3;x)} t dE^{x}_t .$$
By the definition of $\delta$, we obtain that $\lambda (s_1^-; x)\stackrel{\eqref{delta1}}{>} C_1
\pm \delta \stackrel{\eqref{delta3}}{>}  C_2\mp \frac{s_2-s_1}{s_3-s_2}\delta \stackrel{\eqref{delta2}}{>} \lambda (s_3;x)$.
Hence,
$$\lambda ( x_\pm )  =\lambda (x) \pm \left(\delta \chi_{[s_1,s_2)}-  \delta \frac{s_2-s_1}{s_3-s_2}   \chi_{[s_2,s_3)} \right).  $$
% Since $\lambda(y_{\pm})$ is a nonincreasing function, it follows that $\lambda(y_{\pm})=y_+.$ Also, $g_{\pm}\prec f$ for sufficiently small $\delta>0.$ Indeed,
In particular,  for every $s\notin [s_1,s_3]$, we have
$$\int_0^s \lambda (t; x_\pm) dt \le \int_0^s \lambda (t;y) dt . $$
Recall that $f: s\mapsto \int_0^s\lambda (t;y)-\lambda (t;x)dt $ is continuous
and
 so there exists
$ s'\in [s_1, s_2]  $ such that $$  \min _{s\in[s_1,s_2]}\int_0^s \lambda (t;y ) -\lambda (t;x)dt =  f(s') >0 .$$
Hence,
for all $s\in[s_1,s_2]$, there exists a  sufficiently small $\delta$ such that
\begin{align*}
&\quad  \int_0^s \lambda(t; y)-\lambda(t; x _{\pm})dt \\
& =   \int_0^s \lambda (t;y)-\lambda (t;x)dt \pm  \int_0^s \left(\delta \chi_{[s_1,s_2)}-  \delta \frac{s_2-s_1}{s_3-s_2}   \chi_{[s_2,s_3)} \right) ds \\
&\ge \min_{s\in[s_1,s_2]}\int_0^s  \lambda (t;y)-\lambda (t;x) dt \pm  \int_0^s \left(\delta \chi_{[s_1,s_2)}-  \delta \frac{s_2-s_1}{s_3-s_2}   \chi_{[s_2,s_3)} \right) ds \\
&\ge \min  _{s\in[s_1,s_2]}\int_0^s  \lambda (t;y)-\lambda (t;x) dt  -  2\delta \cdot  (s_2-s_1) \\
&   = f(s')  - 2\delta\cdot  (s_2-s_1)   > 0 .
\end{align*}
On the other hand, 
since $\int_0^{s_3}\lambda (x) \ge \int_0^{s_3}\lambda (x) =\int_0^{s_3} \lambda (x_\pm)$, $t\mapsto \int_0^t\lambda (y)$ is a concave function and $\lambda (x_\pm)$ is a constant on $(s_2,s_3)$, it follows 
 that $x_\pm \prec y$.
Hence, $x$ is not an extreme point of $\Omega(y)$, which  is a contradiction.
%Therefore, in this
%This means that $A=\emptyset,$ and therefore, we have $\lambda(g)=\lambda(f).$

\emph{Case 2.}
Suppose that $\lambda(x) $ takes only two   values  on $(a_i,b_i)$.
 Hence, there exist $a_i <s_0 < b_i $ such that $\lambda (x)|_{[a_i,s_0)}= C_1$ and  $\lambda (x)|_{[s_0, b_i  )}= C_2$, $C_1>C_2$.
 Let $C' _1 := \frac{1}{s_0-a_i} \int_{a_i}^{s_0} \lambda (t;y)dt$.
 By the assumption that $ \int_{a_i}^{s } \lambda (t;x )dt <\int_{a_i }^s \lambda (t;y)dt$, $s\in ({a_i},b_i )$, we have $C'_1 > C_1$.
 Let $C_2' := \frac{1 }{b_i-s_0}\int_{s_0}^{b_i}  \lambda (t;y) dt  $.
Recall that $\int_{a_i}^{b_i}\lambda(t;y)-\lambda (t;x)dt=0$ and $  \int_{a_i }^{s_0 } \lambda (t;y)-\lambda (t;x )dt >0$.
  We obtain that $\int_{s_0}^{b_i}\lambda(t;y)-\lambda (t;x)dt=\int_{a_i}^{b_i}\lambda(t;y)-\lambda (t;x)dt - \int_{a_i }^{s_0 } \lambda (t;y)-\lambda (t;x )dt <0$.  Hence,
 $$C_2' = \frac{1 }{b_i-s_0}\int_{s_0}^{b_i}  \lambda (t;y) dt   <\frac{1 }{b_i-s_0}\int_{s_0}^{b_i}  \lambda (t;x) dt =C_2.$$
 Let $\delta >0$ be such that  $$\delta < \min\{ C_1'-C_1 , (C_1-C_2)\frac{b_i-s_0}{b_i-a_i} ,  (C_2-C_2' ) \frac{ b_i -s_0} {s_0-a_i} , \lambda(a_i^-;x)-C_1, (C_2-\lambda(b_i;x))\frac{b_i-s_0}{s_0-a_i} \}.   $$
Set $$u:=E^x\{ C_1\}-  \frac{s_0-a_i}{ b_i -s_0} \cdot   E^x \{ C_2\}$$
  and $x_\pm := x \pm  \delta u$.
Note  that $\lambda (x)$ can be represented in the form of  $$\int_{C_1^+} ^\infty t dE^{x}_t +C_1E^x\{C_1\} +C_2E^x\{C_2\} +\int_{-\infty} ^{C_2^- } t dE^{x}_t.$$
Therefore,
$$x_{\pm}=\int_{\lambda(a_i^-;x)} ^\infty t dE^{x}_t +(C_1
\pm \delta)E^x\{C_1\} +(C_2 \mp   \frac{s_0-a_i}{ b_i -s_0} \delta )E^x\{C_2\} +\int_{-\infty} ^{\lambda(b_i;x)  }   t dE^{x}_t .$$
By the definition of $\delta$, we obtain that $\lambda (a_i^-; x)\ge C_1
\pm \delta >  C_2\mp  \frac{s_0-a_i}{ b_i -s_0}\delta \ge \lambda (b_i;x)$.
Hence,
 $$  \lambda (x_\pm) = \lambda (x) \pm \left( \delta \chi_{[a_i ,s_0)} -\delta \frac{s_0-a_i}{ b_i -s_0}  \chi_{[s_0,b_i )} \right)$$
 and
  $\int_0^1 \lambda ( t;x_{\pm}) dt = \int_0^1 \lambda (t; x) dt = \int_0^1 \lambda (t; y   ) dt  $.
Recall that every partial averaging operator $T$\footnote{Let $\cA=\{A_k\subset (0,1)\}$ be a finite or infinite sequence of disjoint sets  of positive  measure. Denote by $A_\infty$ the complement of $\cup_k A_k$. The partial averaging operator is defined by (see \cite[p. 198]{SZ} or \cite[Chapter 1.2.12]{Zanin})
$$ P(x\mid \cA) =\sum_k \frac{1}{m(A_k)} \left( \int_{A_k} x(s) ds\right)\chi_{A_k} +x \chi_{A_\infty}, ~x\in L_1(0,1).   $$ }
 is a doubly stochastic operator (see e.g. \cite[Chapter 1.2.12]{Zanin} or \cite[Section 2]{Chong}), i.e., $T f \prec f$, $f\in L_1(0,1)$ (see \cite{R} for definition of doubly stochastic operators).
Hence,  for every $s\in [a_i ,s_0]$, we have
  $$\int_{a_i}^s \lambda (t;x_{\pm})dt  = \int_{a_i}^s \Big(C_1  \pm \delta \Big) dt< \int_{a_i}^s C_1' dt  \le  \int_{a_i}^s \lambda (t;y )dt. $$
Similarly, for every $s\in [s_0,b_i ]$, we have
\begin{align*}
\int_s^{b_i}  \lambda(t;y) dt & \le \int_s^{b_i}  C_2' dt
  <  \int_s^{b_i}   \Big( C_2 -\frac{s_0-a_i}{ b_i -s_0} \delta  \Big)  dt \le  \int_s^{b_i}   \Big( C_2 \mp \frac{s_0-a_i}{ b_i -s_0} \delta  \Big)  dt = \int_s^{b_i}   \lambda (t;x _\pm)dt \end{align*}
and therefore, for every $s\in [s_0,b_i ]$, we have
\begin{align*}
 \int_{a_i}^s  \lambda(t;y)-\lambda(t; x_{\pm }) dt  &= -  \int_s^{b_i}  \lambda(t;y)-\lambda(t; x _{\pm }) dt  >0. % \\
   % &= \int_0^{1-s} \lambda((1 -t)-;x) -\lambda((1 -t)-;y_\pm)dt  > 0,
\end{align*}
%i.e., $ \int_0 ^s (\lambda(f)-\lambda(g_{+}))(t)dt>0 $.
%Hence, letting $$\delta < \min\left\{ - \frac{1-s_0}{s_0}B_2,  B_1  , \frac{C_1-C_2}{2}\right\}, $$
%we have $g_{+ }\prec f$.
%For the case of $g_-$, one only need to observe that
 %$$ \int_0^s g(t)-  g_-(t)dt =\int_0^s u(t) dt=\int_0^s \delta  \chi_{[0,s_0) }-  \frac{s_0}{1-s_0} \delta  \chi_{[s_0,1]}   dt > 0   , ~s \in (0,1) ,$$
%and therefore, for every $s \in (0,1)$,
%$$\int_0^s(\lambda(f)-\lambda(g_{-}))(t)dt=\int_0^s(f-g_-)(t)dt\geqslant\int_0^s(f-g)(t)dt>0.$$
Hence, $x_\pm \prec y $ and thus  $x$ is not an extreme point of $\Omega(y),$ which  is a contradiction. %Therefore $\lambda(x)=\lambda(y).$

\emph{Case 3. }
Now, we consider the case when  $\lambda (x)|_{(a_i,b_i )}=C.$
Assume that $E^x\{C\}$ is not a minimal projection.
Let $E^{x} \{C\} =p_1+p_2$, $p_1\perp p_2  \in \cP(\cM )$.
Let $\delta>0$.
Set $$u:= p_1-   \frac{\tau(p_1)}{\tau(p_2)} p_2 $$  and
$$x_{\pm}:=x \pm  \delta u.$$
It is clear that $\int_0^1 \lambda (t;x_{\pm}) dt = \int_0^1 \lambda (t; x) dt = \int_0^1 \lambda(t;y) dt  $. %and  $\lambda (g_+) =g_+ =\lambda (g_-)$.
%Hence, it suffices to show that
The same argument in Case 2 yields that there exists  sufficiently small $\delta$ such that $x_\pm  \prec y$.
That is, $x$ is not an extreme point of $\Omega (y)$. Hence, $E^x\{C\}$ is a minimal projection in $
\cM$.
Moreover, by the assumption that $a_i,b_i\notin A$, we have
$\int_0^{a_i}\lambda (t;y) - \lambda(t;x)dt=0 =\int_0^{b_i}\lambda (t;y) - \lambda(t;x)dt$. Hence,
$\int_{a_i}^{b_i}\lambda (t;x)dt=\int_{a_i}^{b_i}\lambda (t;y)dt$.
which completes the proof.

``$\Leftarrow$''.
The converse implication follows from Lemma \ref{lemma:ex}.
%
%Now, we prove the converse implication. Assume that $\lambda(g)=\lambda(f)$ and we shall prove that $g\in\text{extr}(\Omega(f))$.
%Assume that $g=\frac12(g_1+g_2)$ with $g_i\in \Omega(f)$, $i=1,2$.
%It suffices to prove that $g_1=g_2.$
% By \cite[Theorem 3.3]{C2},  we have
%\begin{eqnarray}\label{eq:g}
%\int_0^s\lambda(t;g)\,dt &=& \frac12\int_0^s\lambda(t;g_1+g_2)\,dt %\nonumber \\
%  &\leqslant& \frac12\left(\int_0^s\lambda(t;g_1)\,dt+\int_0^s\lambda(t;g_2)\,dt\right) \nonumber \\
%  &\leqslant& %\frac12\left(\int_0^s\lambda(t;f)\,dt+\int_0^s\lambda(t;f)\,dt\right) \nonumber  \\
%  &=& \int_0^s\lambda(t;g)\,dt , ~  s\in[0,1).
%\end{eqnarray}
%
%We claim that %$\lambda(g_1+g_2)=\lambda(g_1)+\lambda(g_2)$ and
%$\lambda(g_1)=\lambda(g_2)=\lambda(g).$ Indeed, by $g_2\in \Omega(f)$ %and $\lambda(g)=\lambda(f)$, we have
%$$\int_0^{s}\lambda(t;g_2)\,dt\leqslant\int_0^{s }\lambda(t;g)\,dt,$$
%for all $s\in(0,1)$.
%Suppose that there exist $s_0\in(0,1)$ such that
%$$\int_0^{s_0}\lambda(t;g_1)\,dt<\int_0^{s_0}\lambda(t;g)\,dt. $$
%Then,
%$$\int_0^{s_0}(\lambda(t;g_1)+\lambda(t;g_2))\,dt<2\int_0^{s_0}\lambda(t;g)\,dt$$
%which is a contradiction with \eqref{eq:g}.
%  $$\int_0^s\lambda(t;g)\,dt \le \frac12\left(\int_0^s\lambda(t;g_1)\,dt+\int_0^s\lambda(t;g_2)\,dt\right),\ \ \forall s\in(0,1),$$
%established above.
%Hence,
% $\lambda(g_1)=\lambda(g). $
%Arguing similarly, we obtain that
% $\lambda(g_1)=\lambda(g_2)=\lambda(g).$
%  Hence,
%  $$  \lambda(g_1+g_2)= 2\lambda (g)=\lambda(g_1)+\lambda(g_2).  $$
%By Proposition  \ref{noncomm KPS lemma},
%we obtain that
%$g_1=g_2.$
\end{proof*}

%The following consequence of Theorem \ref{Ryff} was given in \cite[Corollary 29.]{SZ}.
%\begin{cor}\label{SukZanCom} Let $f\in L_1(0,1)_+$, then

%$$\emph{extr}(\Omega_{+}(f))=\left\{g\in L_1(0,1)_+\ :\ %\mu(g)=\mu(f)\right\}$$
%and
%$$\emph{extr}(\Omega'_{+}(f))=\left\{g\in L_1(0,1)_+\ :\ \mu(g)=\mu(f)\chi_{[0,\beta]}\emph{ for some }\beta\geqslant0\right\}.$$

%\end{cor}
\begin{proof*}[of Corollary \ref{corsh}]
Let $\cM$
 be a type II$_1$ factor (resp. $\cM=M_n(\mathbb{C})$) and $\cA$ be a diffuse abelian von Neumann subalgebra of $\cM$  (resp. $\cA$ is the  diagonal masa in $M_n (\mathbb{C})$).
 The first equality in Corollary
\ref{corsh} has been established in \cite{AM} (and the Schur-Horn theorem \cite{Schur,Horn}, see also \cite[(1.1)]{AM}).
Since $\cA$ is diffuse (resp. the diagonal masa in $M_n(\mathbb{C})$), it follows that for any $y\in \cM_h$, there exists $z\in \cA_h$ such that $\lambda (z)=\lambda (y)$.
 Applying  Theorem \ref{RyffNonCom}, we obtain that % for any $z\in \cA$,
 $$ \extr\{x\in \cA_h : x\prec z\} =\{x\in \cA_h : \lambda (x)=\lambda (z)\} . $$
Therefore,
we have
$$ \extr(\{ x\in \cA_h : x\prec y \})
=\extr(\{ x\in \cA_h : x\prec z \})
= \{  x\in \cA_h : \lambda (x)
=\lambda (z)
=\lambda(y)\},$$
which establishes the second equality.
\end{proof*}

{\bf Acknowledgements}
The first author supported by the grant (No. AP08052004 and No. AP08051978) of the Science Committee of the Ministry of Education and Science of the Republic of Kazakhstan.
The second author acknowledges the support of University International Postgraduate Award (UIPA).
The   third   author was supported by the Australian Research Council  (FL170100052).

The authors would like to thank Thomas Schekter and Dmitriy Zanin for helpful discussions
 and thank the reviewers for numerous useful comments and suggestions.

\affiliationone{% in this example, two authors share an institution
D. Dauitbek\\
Abai Kazakh National Pedagogical University, 050010, Almaty, Kazakhstan;\\
Al-Farabi Kazakh National University, 050040, Almaty, Kazakhstan;\\
Institute of Mathematics and Mathematical Modeling, 050010, Almaty, Kazakhstan.
  \email{dostilek.dauitbek@gmail.com}\\}
\affiliationtwo{% in this example, two authors share an institution
   J. Huang and  F. Sukochev\\
  School of Mathematics and Statistics, University of New South Wales, Kensington, 2052, NSW, \\
   Australia
   \email{jinghao.huang@student.unsw.edu.au\\
   f.sukochev@unsw.edu.au}}

\end{document}